%% file: paper_BFBT.tex
\documentclass{article} 

\usepackage[utf8x]{inputenc}

\usepackage[USenglish]{babel}

\usepackage[T1]{fontenc}

\usepackage{amsmath}            %
\usepackage{amssymb}            %
\usepackage{mathrsfs}           %
\usepackage{mathtools}          %
\usepackage[section]{algorithm} %

\usepackage{microtype}          %
\usepackage[normalem]{ulem}     %

\usepackage{geometry}           
\usepackage{adjustbox}          %

\usepackage{graphicx}           %
\usepackage{epstopdf}           %
\usepackage{subcaption}         %
\usepackage{booktabs}           %

\usepackage{color}              %
\usepackage[table]{xcolor}      %
\usepackage{tikz}
  \usetikzlibrary{calc}
  \usetikzlibrary{external} 
\usepackage{pgfplots}
  \pgfplotsset{compat=newest}
\usepackage[outline]{contour}   %
  \contourlength{1.2pt}

\usepackage[group-separator={,}]{siunitx}
  \DeclareSIUnit[number-unit-product = {\,}] \year{yr}

\usepackage{hyperref}           %

\usepackage{algpseudocode}      %
\usepackage{amsthm}             
\usepackage{cleveref}           %
\usepackage{cite}               %

\graphicspath{{graphics/}}

\include{colors_jr}

\let\diag\relax

\include{math_symbols}

\theoremstyle{plain}
\newtheorem{lemma}{Lemma}[section]
\newtheorem{theorem}[lemma]{Theorem}
\newtheorem{corollary}[lemma]{Corollary}

\theoremstyle{definition}

\theoremstyle{remark}
\newtheorem{remark}[lemma]{Remark}

\algblockdefx[ParallelFor]{ParallelFor}{EndParallelFor}%
  [1]{\textbf{parallel for} #1 \textbf{do}}%
  {\textbf{end parallel for}}

\algblockdefx[Atomic]{Atomic}{EndAtomic}%
  [0]{\textbf{atomic}}%
  {\textbf{end atomic}}

\newcommand{\infigure}[1]{\textit{#1}}

\crefname{part}{Part}{Parts}
\crefname{chapter}{Chapter}{Chapter}
\crefname{section}{Section}{Sections}
\crefname{subsection}{Section}{Sections}
\crefname{subsubsection}{Section}{Sections}
\crefname{equation}{Equation}{Equations}
\crefname{figure}{Figure}{Figures}
\crefname{table}{Table}{Tables}
\crefname{definition}{Definition}{Definition}
\crefname{lemma}{Lemma}{Lemmata}
\crefname{theorem}{Theorem}{Theorems}

\hypersetup{
  colorlinks=false,  %
  frenchlinks=false,
  pdfborder={0 0 0},  %
  naturalnames=false,
  hypertexnames=false,
  breaklinks
}

\makeatletter
\newcounter{algorithmicH} %
\let\oldalgorithmic\algorithmic
\renewcommand{\algorithmic}{%
  \stepcounter{algorithmicH} %
  \oldalgorithmic} %
\renewcommand{\theHALG@line}{ALG@line.\thealgorithmicH.\arabic{ALG@line}}
\makeatother

\makeatletter
\def\addlegendimage{\pgfplots@addlegendimage}
\makeatother

\newcommand{\minor}[1]{\textcolor{jr@medgray}{#1}}

\newcommand{\schurMass}[0]{\ensuremath{ \PressMassMat(1/\visc) }}
\newcommand{\BFBT}[0]{\ensuremath{ \text{BFBT} }}
\newcommand{\massBFBT}[0]{\ensuremath{ \VelMassMat\text{-BFBT} }}
\newcommand{\diagBFBT}[0]{\ensuremath{ \diag(\ViscStressMat)\text{-BFBT} }}
\newcommand{\wBFBT}[0]{\ensuremath{ \text{w-BFBT} }}

\newcommand{\HMGwBFBT}[0]{\ensuremath{ \text{HMG}\! + \!\wBFBT }}

\newcommand{\sinkerind}{\ensuremath{\chi}}
\newcommand{\sinkerdecay}{\ensuremath{\delta}}
\newcommand{\sinkerwidth}{\ensuremath{\omega}}
\newcommand{\bdramp}{\ensuremath{a}}

\begin{document}

\title{%
  Weighted BFBT Preconditioner for Stokes Flow Problems
  with Highly Heterogeneous Viscosity%
  \thanks{%
    This research was partially supported by
    NSF grants CMMI-1028889 %
           and ARC-0941678 %
    and
    U.S.\ Department of Energy grants DE-FC02-13ER26128 %
    and
    the Scientific Discovery through
    Advanced Computing (SciDAC) program.
  }
}
\author{%
Johann Rudi%
\thanks{Institute for Computational Engineering and Sciences,
        The University of Texas at Austin, USA
        (\texttt{johann@ices.utexas.edu}, \texttt{omar@ices.utexas.edu}).}%
\and
Georg Stadler%
\thanks{Courant Institute of Mathematical Sciences,
        New York University, USA
        (\texttt{stadler@cims.nyu.edu}).}%
\and
Omar Ghattas%
\footnotemark[2] %
\thanks{Jackson School of Geosciences and Department of Mechanical Engineering,
        The University of Texas at Austin, USA.}%
}

\pagestyle{myheadings}
\thispagestyle{plain}
\markboth{
  Johann Rudi, Georg Stadler, and Omar Ghattas%
}{
  \wBFBT\ Preconditioner for Stokes Flow Problems%
}

\ifpdf
\hypersetup{
  pdftitle={
    Weighted BFBT Preconditioner for Stokes Flow Problems
    with Highly Heterogeneous Viscosity},
  pdfauthor={Johann Rudi, Georg Stadler, and Omar Ghattas},
}
\fi

\maketitle

\begin{abstract}
  We present a weighted \BFBT\ approximation (\wBFBT) to the inverse Schur
  complement of a Stokes system with highly heterogeneous viscosity.
  When used as part of a Schur complement-based
  Stokes preconditioner,
  we observe robust fast convergence %
  for Stokes problems with smooth but
  highly varying (up to 10 orders of magnitude) viscosities,
  optimal algorithmic scalability with respect to mesh refinement,
  and only a mild dependence on the polynomial order of high-order finite
  element discretizations (\discrQPdisc{k}{k-1}, order $k\ge2$).
  For certain difficult problems, we demonstrate numerically that
  \wBFBT\ significantly improves Stokes solver
  convergence over the widely used inverse viscosity-weighted pressure mass
  matrix approximation of the Schur complement.
  In addition, we derive theoretical eigenvalue bounds to prove spectral
  equivalence of \wBFBT.
  Using detailed numerical experiments, we discuss modifications to \wBFBT\ at
  Dirichlet boundaries that decrease the number of iterations.
  The overall algorithmic performance of the Stokes solver is governed
  by the efficacy of \wBFBT\ as a Schur complement approximation and, in
  addition, by our parallel hybrid spectral-geometric-algebraic multigrid (HMG)
  method, which we use to approximate the inverses of the viscous block and
  variable-coefficient pressure Poisson operators within \wBFBT.
  Building on the scalability of HMG, our Stokes solver achieves
  a parallel efficiency of 90\% while weak scaling
  over a more than 600-fold increase from 48 to all 30,000 cores of
  TACC's Lonestar~5 supercomputer.
\end{abstract}



\section{Introduction}
\label{sec:intro}

\subsection{Motivation and governing equations}
\label{sec:motivation}

Many problems in science and engineering involve creeping flows of
non-Newtonian fluids~\cite{GlowinskiXu11, Rajagopal93}.  Important
examples can be found in geophysical fluid flows, where the
incompressible Stokes equations with power-law rheology have become a
prototypical continuum mechanical description for creeping flows
occurring in mantle convection~\cite{SchubertTurcotteOlson01}, magma
dynamics~\cite{McKenzie84},
and ice flow~\cite{Hutter83}.
The linearization of the nonlinear momentum equations, for instance
within a Newton method, leads to
incompressible Stokes-like equations with highly
heterogeneous viscosity fields.

In particular, simulations of Earth's mantle convection at global scale
\cite{StadlerGurnisBursteddeEtAl10} exhibit extreme computational challenges
due to a highly heterogeneous
viscosity stemming from
its dependence on temperature and
strain rate as well as sharp viscosity gradients in narrow regions
modeling tectonic plate boundaries (six orders of magnitude drop in
$\sim$\SI{5}{\kilo\meter})
\cite{BursteddeGhattasGurnisEtAl08,RudiMalossiIsaacEtAl15}.  This
leads to a wide range of spatial scales since small localized features
at plate boundaries of size $\LandauO(\SI{1}{\kilo\meter})$ influence
plate motion at continental scales of
$\LandauO(\SI{1000}{\kilo\meter})$.
The complex character of the flow presents severe computational
challenges for iterative solvers due to poor conditioning of
the linear systems that arise.

Since we focus here on preconditioning linearized Stokes-like problems
that arise at each step of a nonlinear solver, we can simplify our
problem setup by
taking the viscosity to be independent of the strain rate,
but otherwise exhibiting severe spatial heterogeneity.
Given a bounded domain
  $\Omega\subset\R^d$, $d\in\{2,3\}$,
right-hand side forcing
  $\rhsvel(\bx)$,
and spatially-varying viscosity
  $\visc(\bx)\ge \viscmin > 0$ for all $\bx\in\Omega$,
we consider the incompressible Stokes equations
with homogeneous Dirichlet boundary conditions
\begin{subequations}
\label{eq:stokes}
\begin{alignat}{2}
  \label{eq:momentum}
  - \divergence
  \bigl[ \visc(\bx) \, (\gradient\vel + \gradient\vel\transpose) \bigr]
  + \gradient\press &= \rhsvel
  &&\quad\text{in }\Omega,
  \\
  \label{eq:mass}
  - \divergence\vel &= 0
  &&\quad\text{in }\Omega,
  \\
  \label{eq:dir-bc}
  \vel &= 0
  &&\quad\text{on }\partial\Omega,
\end{alignat}
\end{subequations}
where $\vel$ and $\press$ are the unknown velocity and pressure fields,
respectively.%
\footnote{%
  Note that Newton linearization (as opposed to, e.g., Picard) results
  in a fourth-order anisotropic viscosity tensor, with the anisotropic
  component proportional to the tensor product of the second-order strain
  rate tensor with itself. This anisotropic viscosity tensor also
  appears in adjoint Stokes equations, which are important for inverse
  problems \cite{PetraZhuStadlerEtAl12, IsaacPetraStadlerEtAl15,
    WorthenStadlerPetraEtAl14}. For power-law rheologies that are
  typical of many geophysical flows, the anisotropic component of the
  viscosity tensor is dominated by the isotropic component
  $\visc(\bx)\IdMat$ \cite{IsaacStadlerGhattas15,
    RudiMalossiIsaacEtAl15}. Thus for many forward and inverse modeling
  applications, the performance of our preconditioner for isotropic
  viscosities is indicative of its performance for operators with
  anisotropic viscosity tensors.  Hence we focus on the former in this
  paper.
}

\subsection{Discretization and computational challenges}
\label{sec:discretization}

Discretizing \eqref{eq:stokes} leads to
a linear algebraic system of equations of the form
\begin{equation}
  \label{eq:discrete-stokes}
  \begin{bmatrix}
    \ViscStressMat &\GradMat \\ \DivMat &\ZeroMat
  \end{bmatrix}
  \begin{bmatrix}
    \velvec \\ \pressvec
  \end{bmatrix}
  =
  \begin{bmatrix}
    \rhsvelvec \\ \zerovec
  \end{bmatrix},
\end{equation}
where $\ViscStressMat$, $\DivMat$, and $\GradMat$ are matrices
corresponding to discretizations of the viscous stress, divergence,
and gradient operators, respectively.  The discretization is carried
out by high-order finite elements on (possibly aggressively adaptively
refined) hexahedral meshes with velocity--pressure pairings
\discrQPdisc{k}{k-1} of polynomial order $k\ge2$ with a continuous,
nodal velocity approximation \discrQ{k} and a discontinuous, modal
pressure approximation \discrPdisc{k-1}.  These pairings yield optimal
asymptotic convergence rates of the finite element approximation to
the infinite-dimensional solution with decreasing mesh element size,
are inf-sup stable on
general, non-conforming hexahedral meshes with ``hanging nodes,''
and have the advantage of preserving mass locally at the element level due to
the discontinuous pressure \cite{ElmanSilvesterWathen14, HeuvelineSchieweck07,
StenbergSuri96}.
While these properties have been recognized to be important for geophysics
applications (e.g., see~\cite{MayBrownLePourhiet14, MayBrownPourhiet15}),
the high-order discretization, adaptivity, and discontinuous pressure
approximation present significant additional difficulties for
iterative solvers (relative to low order, uniform grid, continuous
discretizations). Finally, a number of frontier geophysical problems,
such as global mantle convection with plate boundary-resolving meshes
and continental ice sheet models with grounding line-resolving meshes,
result in billions of degrees of freedom, demanding efficient
execution and scalability on
leading edge supercomputers
\cite{BursteddeGhattasGurnisEtAl08,
RudiMalossiIsaacEtAl15, StadlerGurnisBursteddeEtAl10}.

The applications we target (such as global mantle convection) exhibit
all of the difficulties described above (severe heterogeneity, very
large scale, need for aggressively-adapted meshes, need for high
order, mass-conserving discretization) and thus demand robust
and effective preconditioners for~\eqref{eq:discrete-stokes},
resulting in iterative solvers with optimal (or nearly optimal)
algorithmic and parallel scalability. This paper describes the design
of such a preconditioner and its analysis and performance evaluation
for problems with highly heterogeneous viscosity. The
preconditioner---{which we call \em weighted BFBT} (\wBFBT)---is of
Schur complement type, and we study its robustness as well as its
algorithmic and parallel scalability.

\subsection{Iterative methods and Schur complement approximations}
\label{sec:schur-approx}

An effective approximation of the Schur\footnote{%
  Strictly speaking, our definition is the negative Schur
  complement. However, as in
  \cite{ElmanSilvesterWathen14}, we prefer to work with
  positive-definite operators and thus define the Schur complement to
  be positive rather than negative definite.} complement
$\SchurMat\coloneqq\DivMat\ViscStressMat^{-1}\GradMat$ is an essential
ingredient for attaining fast convergence of Schur complement-based
iterative solvers for \eqref{eq:discrete-stokes}.
More precisely, a sufficiently good approximation of the inverse Schur
complement
  $\SchurPcMat^{-1}\approx\SchurMat^{-1}$
is sought, which, together with an approximation of the
inverse viscous block,
  $\ViscStressPcMat^{-1}\approx\ViscStressMat^{-1}$,
is used in an iterative scheme with right preconditioning based on an
upper triangular block matrix:
\begin{equation}
  \label{eq:precond-stokes}
  \begin{bmatrix}
    \ViscStressMat &\GradMat \\ \DivMat &\ZeroMat
  \end{bmatrix}
  \begin{bmatrix}
    \ViscStressPcMat &\GradMat \\ \ZeroMat &\SchurPcMat
  \end{bmatrix}^{-1}
  \begin{bmatrix}
    \tilde\velvec \\ \tilde\pressvec
  \end{bmatrix}
  =
  \begin{bmatrix}
    \rhsvelvec \\ \zerovec
  \end{bmatrix}.
\end{equation}
Note that the original solution to \eqref{eq:discrete-stokes} is recovered by
applying the preconditioner once to the solution of \eqref{eq:precond-stokes}.
For the preconditioned Stokes system~\eqref{eq:precond-stokes}, we use GMRES as
the Krylov subspace solver.
This particular combination of Krylov method and preconditioner is known to
converge in just two iterations for
exact
choices of $\ViscStressPcMat^{-1}$
and $\SchurPcMat^{-1}$ \cite{BenziGolubLiesen05}.

The most widely used
approximation of the Schur complement (for variable viscosity Stokes systems)
is the inverse viscosity-weighted mass matrix of the pressure space
\cite{BursteddeStadlerAlisicEtAl13, IsaacStadlerGhattas15,
KronbichlerHeisterBangerth12, MayBrownPourhiet15},
denoted by
  $\PressMassMat(1/\visc)$,
with entries
  $\left[\PressMassMat(1/\visc)\right]_{i,j} =
    \int_\Omega \presstest_i(\bx) \, \presstest_j(\bx) / \visc(\bx) \intd\bx$,
where
  $\presstest_i,\presstest_j \in \discrPdisc{k-1}$
are global basis functions of the finite dimensional space \discrPdisc{k-1}.
Since the basis functions of \discrPdisc{k-1} are modal
and not orthogonal to each other, the mass matrix is not diagonal, and thus
$\PressMassMat(1/\visc)$ is typically diagonalized to further simplify its
inversion.  One common way to obtain a diagonalized version is mass lumping.
For nodal discretizations, the corresponding diagonal elements are
computed by summation of the entries of each matrix
row, i.e., $\PressMassMat(1/\visc)\onevec$, where $\onevec$ is the vector with
ones in all entries.
For modal discretizations,
we generalize the lumping procedure by using
the coefficient vector, $\onevec_{\{\presstest_i\}_{i}}$, representing
the constant function having value $1$
in the associated basis $\{\presstest_i\}_{i}$, i.e.,
\begin{equation}
  \label{eq:lumping}
  \PressMassLumpMat(1/\visc) \coloneqq
    \diag(\PressMassMat(1/\visc)\onevec_{\{\presstest_i\}_{i}}).
\end{equation}
Provided that $\visc$ is sufficiently smooth, \schurMass\ can be an effective
approximation of $\SchurMat$ in numerical experiments
\cite{BursteddeGhattasStadlerEtAl09} and spectral
equivalence can be shown \cite{GrinevichOlshanskii09}.
However, it has been observed in applications with highly heterogeneous
viscosities $\visc$ (e.g., mantle convection
\cite{MayMoresi08,RudiMalossiIsaacEtAl15}) that convergence slows down
significantly due to a poor Schur complement approximation by \schurMass.
Therefore, we propose a new approximation, \wBFBT, that remains robust
when \schurMass\ fails.

Preconditioners based on BFBT approximations for the Schur complement
were initially proposed in
\cite{Elman99} %
for the Navier--Stokes equations.
Over the years, these ideas were refined and extended
\cite{SilvesterElmanKayEtAl01,KayLoghinWathen02,ElmanHowleShadidEtAl06,
  ElmanHowleShadidEtAl08,ElmanTuminaro09}
to arrive at a class of closely related Schur complement approximations:
Pressure Convection--Diffusion, BFBT, and Least Squares Commutator.
The underlying principle, now in a Stokes setting, is that one seeks a
commutator matrix
  $\mat{X}$
such that the following commutator nearly vanishes,
\begin{equation}
  \label{eq:discr-commutator}
    \ViscStressMat\BfbtScalRightMat^{-1}\GradMat - \GradMat\mat{X}
  \approx
    \ZeroMat,
\end{equation}
for a given diagonal matrix $\BfbtScalRightMat^{-1}$.
The Navier--Stokes case differs from Stokes in that the viscous stress
matrix $\ViscStressMat$ contains an additional convection term.  The
motivation for seeking a near-commutator $\mat{X}$ is
that~\eqref{eq:discr-commutator} can be rearranged by multiplying
\eqref{eq:discr-commutator} with $\DivMat\ViscStressMat^{-1}$ from
the left and, provided the inverse exists, with $\mat{X}^{-1}$ from the right
to obtain
  $\SchurMat %
    \approx
    \DivMat\BfbtScalRightMat^{-1}\GradMat\mat{X}^{-1}$,
where the closer
the commutator is to zero, the more accurate the approximation
\cite{ElmanSilvesterWathen14}.
The goal of finding a vanishing commutator can be recast as solving the
following least-squares minimization problem:
\begin{equation}
  \label{eq:lsc-min}
  \text{Find matrix } \mat{X} \text{ minimizing}\quad
    \norm{
      \ViscStressMat\BfbtScalRightMat^{-1}\GradMat \unitvec_j -
      \GradMat\mat{X} \unitvec_j
    }^2_{\BfbtScalLeftMat^{-1}}
  \quad\text{for all } j,
\end{equation}
where
  $\unitvec_j$
is the $j$-th Cartesian unit vector and
the norm arises from a symmetric and positive definite matrix
  $\BfbtScalLeftMat$.
The solution is given by
  $\mat{X} =
    (\DivMat\BfbtScalLeftMat^{-1}\GradMat)^{-1}
    (\DivMat\BfbtScalLeftMat^{-1}\ViscStressMat\BfbtScalRightMat^{-1}\GradMat)$.
Then the BFBT approximation of the inverse Schur complement is derived by
algebraic rearrangement of the commutator \eqref{eq:discr-commutator}:
\begin{equation}
  \label{eq:bfbt}
  \SchurPcMat_{\BFBT}^{-1} \coloneqq
    \left(\DivMat\BfbtScalLeftMat^{-1}\GradMat\right)^{-1}
    \left(
      \DivMat\BfbtScalLeftMat^{-1}\ViscStressMat\BfbtScalRightMat^{-1}\GradMat
    \right)
    \left(\DivMat\BfbtScalRightMat^{-1}\GradMat\right)^{-1}.
\end{equation}
In the literature cited above (which addresses preconditioning of
Navier--Stokes equations with constant viscosity), the diagonal weighting
matrices are chosen as
  $\BfbtScalLeftMat = \BfbtScalRightMat = \VelMassLumpMat$,
i.e., a diagonalized version of the velocity space mass matrix; hence we call
this the \massBFBT\ approximation of the Schur complement.
\massBFBT\ can be used for Stokes problems with constant viscosities
providing convergence similar to \schurMass.
However, the computational cost of applying~\eqref{eq:bfbt} is significantly
higher than the (cheap) application of a possibly diagonalized inverse of
\schurMass.
Moreover, \massBFBT\ is not an option for heterogeneous viscosities because
convergence becomes extremely slow or stagnates, as observed
in~\cite{MayMoresi08}.
Instead, in~\cite{MayMoresi08} for finite element and
in~\cite{FuruichiMayTackley11} for staggered grid finite difference
discretizations, a re-scaling of the discrete Stokes system~\eqref{eq:stokes}
was performed, which essentially alters the diagonal weighting matrices
$\BfbtScalLeftMat$, $\BfbtScalRightMat$.
By choosing entries from $\ViscStressMat$ for these weighting matrices,
it was possible to demonstrate improved convergence with \BFBT\ compared
to \schurMass\ for certain benchmark problems with strong viscosity variations.
Building on ideas from~\cite{MayMoresi08},
\cite{RudiMalossiIsaacEtAl15} chose the weighting matrices such that
$\BfbtScalLeftMat = \BfbtScalRightMat = \diag(\ViscStressMat)$, which
led to superior performance compared to \schurMass\ for highly
heterogeneous mantle convection problems.  Hence we refer to this
approach as \diagBFBT.

However, even \diagBFBT\ can
fail to achieve fast convergence for
some problems/discretizations, as shown below (\cref{sec:comparison}).
Moreover, choosing the weighting matrices as $\diag(\ViscStressMat)$
is problematic for high-order
discretizations, where $\diag(\ViscStressMat)$ becomes a poor
approximation of $\ViscStressMat$.
These drawbacks
lead us to propose the following \wBFBT\ approximation for the inverse Schur
complement:
\begin{equation}
  \label{eq:wbfbt}
  \SchurPcMat_{\wBFBT}^{-1} \coloneqq
    \left(\DivMat\WbfbtScalLeftMat^{-1}\GradMat\right)^{-1}
    \left(
      \DivMat\WbfbtScalLeftMat^{-1}\ViscStressMat\WbfbtScalRightMat^{-1}\GradMat
    \right)
    \left(\DivMat\WbfbtScalRightMat^{-1}\GradMat\right)^{-1},
\end{equation}
where
  $\WbfbtScalLeftMat = \VelMassLumpMat(\bfbtweightleft)$ and
  $\WbfbtScalRightMat = \VelMassLumpMat(\bfbtweightright)$
are lumped velocity space mass matrices (lumping analogously to
\eqref{eq:lumping})
that are weighted by the square root of the viscosity,
  $\bfbtweightleft(\bx) = \sqrt{\visc(\bx)} = \bfbtweightright(\bx)$,
$\bx\in\Omega$.

\subsection{Outline and summary of key results}
\label{sec:outline}

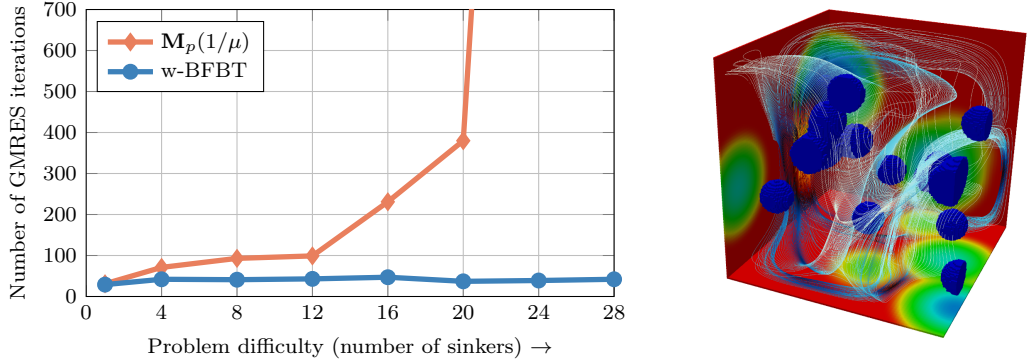
\begin{figure}\centering
  \begin{tikzpicture}[baseline={(current bounding box.center)}]
    \tikzset{
      graph/.style={
        color=#1, line width=2pt},
    }
    \pgfplotsset{
      commonaxis/.style={
        compat=newest,
        width=8.6cm,
        height=5.4cm,
        grid=major,
        tick label style={font=\footnotesize},
        title style={align=center, font=\footnotesize},
        xlabel style={align=center, font=\footnotesize},
        ylabel style={align=center, font=\footnotesize},
        legend style={at={(0.02,0.96)}, anchor=north west, font=\footnotesize},
        legend cell align=left},
      sinkeraxis/.style={
        xmin=0,
        xmax=28,
        ymin=0,
        ymax=700,
        xtick={0,4,8,12,16,20,24,28},
        ytick={0,100,200,300,400,500,600,700} },
      sinkeraxisscale/.style={
        /pgfplots/xmode=linear,
        /pgfplots/ymode=linear},
      sinkeraxisscale/.belongs to family=/pgfplots/scale,
    }
    \def\colorA{jr@div-medred!90!jr@div-darkred}
    \def\colorB{jr@div-medblue!40!jr@div-darkblue}
    \def\markA{diamond*}
    \def\markB{*}
    \def\datadir{graphics/plotdata/classification}

    \def\dataA{\datadir/Tvary_viscDrMagn8_schurMass.txt}
    \def\dataB{\datadir/Tvary_viscDrMagn8_schurBFBT23.txt}
    \begin{axis}[
      commonaxis, sinkeraxis, sinkeraxisscale,
      title={},
      xlabel={Problem difficulty (number of sinkers) $\rightarrow$},
      ylabel={Number of GMRES iterations},
    ]
      \addplot[graph=\colorA,mark=\markA] table [x=sinkers,y=iter] {\dataA};
      \addplot[graph=\colorB,mark=\markB] table [x=sinkers,y=iter] {\dataB};
      \legend{\schurMass, \wBFBT};
    \end{axis}
  \end{tikzpicture}
  \hspace{2em} 
  \begin{adjustbox}{valign=c}
    \includegraphics[width=4.4cm]{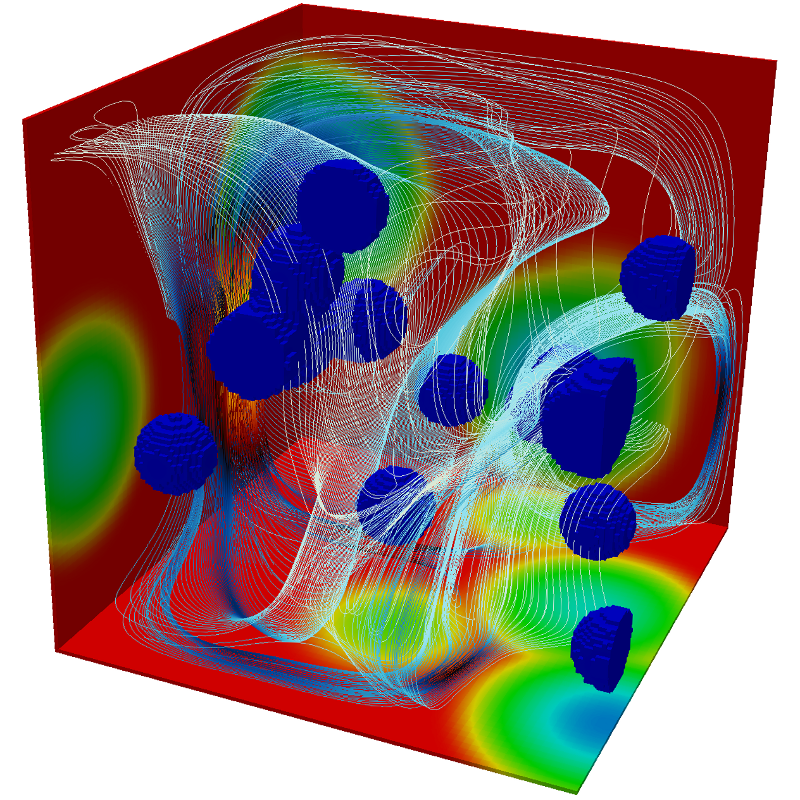}
  \end{adjustbox}
  \caption{%
    \infigure{Left image} shows the improvement in convergence obtained with
    the proposed \wBFBT\ preconditioner over
    a preconditioner using the inverse viscosity-weighted pressure
    mass matrix as Schur
    complement approximation.
    The number of randomly placed sinkers
    (high viscosity inclusions in low viscosity medium) increases
    along the horizontal axis. The vertical axis depicts
    the number of GMRES iterations required for $10^{6}$ residual reduction
    for the
    most popular
    \schurMass\ and the proposed \wBFBT\ preconditioner.
    Fixed problem parameters are the dynamic ratio
    $\dynamicratio(\visc)=\max(\visc)/\!\min(\visc)=10^8$,
    discretization order $k=2$, and the mesh refinement level
    $\ell=7$, resulting in $128^3$ finite elements.
    \infigure{Right image} shows an example viscosity field with 16 sinkers
    (\infigure{blue spheres} depict highly viscous centers of Gaussian-like
    sinkers, low viscosity medium in \infigure{red color}) and the
    \infigure{streamlines} of the computed velocity field.
  }
  \label{fig:conv-improvement}
\end{figure}

After defining a class of benchmark problems (\cref{sec:benchmark}), we compare
the convergence obtained with different Schur complement approximations to
motivate preconditioning with \wBFBT\ (\cref{sec:comparison}).
Theoretical estimates for spectral equivalence of \wBFBT\ are derived in
\cref{sec:theory}.
This is followed by a detailed numerical study showing when \wBFBT\ is
advantageous over \schurMass\ (\cref{sec:robustness}),
and a discussion of boundary modifications for \wBFBT\ that accelerate
convergence (\cref{sec:boundary-numerics}).
In \cref{sec:hmg} we describe an algorithm for \wBFBT-based
Stokes preconditioning, which uses hybrid spectral-geometric-algebraic
multigrid (HMG). Finally, in \cref{sec:scalability}, we provide numerical evidence
for near-optimal algorithmic and parallel scalability. In particular,
we demonstrate that the preconditioner's parallel efficiency
remains high when weak scaling out to tens of thousands of threads
and even millions of threads.%
\footnote{%
  See \cite{RudiMalossiIsaacEtAl15}, which demonstrated parallel scalability
  for a \BFBT-type method.  \wBFBT\ deviates from \cite{RudiMalossiIsaacEtAl15}
  because of a different choice of diagonal weighting matrices in
  \eqref{eq:bfbt}.  However, the work per application of both preconditioners
  is the same, resulting in comparable parallel scalability.
}

To motivate our study of \wBFBT, we give an example for a possible improvement
in convergence in \cref{fig:conv-improvement}.  There, a comparison is drawn
between the \schurMass\ and \wBFBT\ approximations for the Schur complement.
The Stokes problem that is being solved is the multi-sinker benchmark problem
from \cref{sec:benchmark}.  The difficulty of the problem can be increased by
adding more and more high-viscosity inclusions, called sinkers, into a
low-viscosity background medium, which, as a result, introduces more variation
in the viscosity.
As can be seen in the figure, the number of GMRES iterations remain flat when
preconditioning with \wBFBT, whereas the number of GMRES iterations increases
significantly for higher sinker counts when using \schurMass, rendering
\schurMass\ inefficient for these types of difficult problems.  Therefore we
propose \wBFBT\ as an alternative Schur complement approximation for Stokes
flow problems with a highly varying viscosity.

\section{Benchmark problem and comparison of Schur complement approximations}
\label{sec:benchmark-and-comparison}

This section further motivates the need for more effective Schur complement
preconditioners.  We first present a class of benchmark problems that
range from relatively mild viscosity variations to highly heterogeneous.
Then a challenging problem is used to compare Stokes solver convergence with
different Schur complement approximations to demonstrate the limitations of
established methods and motivate the development of \wBFBT.

\subsection{Multi-sinker benchmark problem}
\label{sec:benchmark}

The design of suitable benchmark problems is critical to conduct studies that
can give useful convergence estimates for challenging applications.
We seek complex geometrical structures in the viscosity that generate
irregular, nonlocal, multiscale flow fields.  Additionally, the viscosity
should exhibit sharp gradients and its dynamic ratio
  $\dynamicratio(\visc) \coloneqq \max(\visc)/\!\min(\visc)$
(also commonly referred to as viscosity contrast) can be six orders of
magnitude or higher in demanding applications.
As in \cite{MayBrownLePourhiet14}, we use a multi-sinker test problem
with randomly positioned inclusions (e.g., as in
\cref{fig:conv-improvement}, \infigure{right image}) %
to study solver performance. We find that the arising viscosity
structure is a suitable test for challenging, highly heterogeneous
coefficient Stokes problems, and that the solver performance observed
for such models can be indicative of the performance for other
challenging applications.

In the (open) unit cube domain
  $\Omega = (0,1)^3$,
we define the viscosity coefficient
  $\visc(\bx) \in [\viscmin,\viscmax]$,
  $\bx\in\Omega$,
  $0<\viscmin<\viscmax<\infty$,
with dynamic ratio
  $\dynamicratio(\visc) = \viscmax/\viscmin$
by means of rescaling a $C^{\infty}$ indicator function
  $\sinkerind_{n}(\bx) \in [0,1]$
that accumulates $n$ sinkers via a product of modified Gaussian functions:
\begin{align*}
  \visc(\bx) &\coloneqq
    (\viscmax - \viscmin) (1 - \sinkerind_{n}(\bx)) + \viscmin,
  \quad\bx\in\Omega,
  \\
  \sinkerind_{n}(\bx) &\coloneqq
    \prod_{i=1}^{n}
      1 - \exp\left( \!\!-\sinkerdecay
                     \max\left(0, \abs{\boldvar{c}_i - \bx} -
                                  \frac{\sinkerwidth}{2}\right)^2
              \right),
  \quad\bx\in\Omega,
\end{align*}
where
  $\boldvar{c}_i\in\Omega$, $i=1,\ldots,n$,
are the centers of the sinkers,
  $\sinkerdecay>0$
controls the exponential decay of the Gaussian smoothing, and
  $\sinkerwidth\ge0$
is the diameter of a sinker where $\viscmax$ is attained.
Since all sinkers are equal in size, inserting more of them inside the domain
will eventually result in overlapping with each other and possible
intersections with the domain's boundary.
Throughout the paper, we fix
  $\sinkerdecay=200$,
  $\sinkerwidth=0.1$,
and use the same set of precomputed random points
  $\boldvar{c}_i$
in all numerical experiments.  Two parameters are varied:
(i) the number of sinkers $n$ at random positions (the label S$n$-rand
indicates a multi-sinker problem with $n$ randomly positioned sinkers)
and
(ii) the dynamic ratio $\dynamicratio(\visc)$ which in turn determines
  $\viscmin\coloneqq\dynamicratio(\visc)^{-1/2}$
and
  $\viscmax\coloneqq\dynamicratio(\visc)^{1/2}$.
The right-hand side of~\eqref{eq:discrete-stokes},
  $\rhsvel(\bx) \coloneqq \left(0, 0, \beta(\sinkerind_{n}(\bx) - 1)\right)$,
  $\beta=10$ constant,
is such that it forces the high-viscosity sinkers downward,
similarly to a gravity that pulls on high-density inclusions within a medium of
lower density.

\subsection{Comparison of Schur complement approximations}
\label{sec:comparison}

\begin{figure}\centering
  \begin{tikzpicture}
    \tikzset{
      convgraph/.style={
        color=#1,
        line width=2pt},
    }
    \pgfplotsset{
      convaxis/.style={
        compat=newest,
        width=4.1cm,
        height=3.7cm,
        xmin=0,
        xmax=100,
        ymin=1e-6,
        ymax=5e+0,
        xtick={0,25,50,75,100},
        ytick={1e0,1e-2,1e-4,1e-6},
        grid=major,
        tick label style={font=\footnotesize},
        title style={align=center, font=\footnotesize, yshift=-0.2cm},
        xlabel style={align=center, font=\footnotesize},
        ylabel style={align=center, font=\footnotesize},
        legend style={at={(1.15,1.0)}, anchor=north west, font=\footnotesize} },
    }
    \pgfplotsset{
      convaxisscale/.style={
        /pgfplots/xmode=linear,
        /pgfplots/ymode=log},
      convaxisscale/.belongs to family=/pgfplots/scale,
    }

    \def\colorA{jr@red}
    \def\colorB{jr@purple}
    \def\colorC{jr@blue}
    \def\colorD{jr@green}
    \def\datadir{graphics/plotdata/comparison_levelVary}

    \def\dataA{\datadir/T16rand_viscDrMagn8_schurMass_k2_l4.txt}
    \def\dataB{\datadir/T16rand_viscDrMagn8_schurMass_k2_l5.txt}
    \def\dataC{\datadir/T16rand_viscDrMagn8_schurMass_k2_l6.txt}
    \def\dataD{\datadir/T16rand_viscDrMagn8_schurMass_k2_l7.txt}
    \begin{axis}[
      name=levelVary-S16rand-schurMass,
      convaxis, convaxisscale,
      title={\schurMass},
      ylabel={fixed order $k=2$ \\[4ex]
              $\norm{\text{res}}/\norm{\text{init res}}$},
    ]
      \addplot[convgraph=\colorB] table [x=iter,y=residual] {\dataB};
      \addplot[convgraph=\colorC] table [x=iter,y=residual] {\dataC};
      \addplot[convgraph=\colorD] table [x=iter,y=residual] {\dataD};
    \end{axis}

    \def\dataA{\datadir/T16rand_viscDrMagn8_schurBFBT40_bdrAmp1-1_k2_l4.txt}
    \def\dataB{\datadir/T16rand_viscDrMagn8_schurBFBT40_bdrAmp1-1_k2_l5.txt}
    \def\dataC{\datadir/T16rand_viscDrMagn8_schurBFBT40_bdrAmp1-1_k2_l6.txt}
    \def\dataD{\datadir/T16rand_viscDrMagn8_schurBFBT40_bdrAmp1-1_k2_l7.txt}
    \begin{axis}[
      at=(levelVary-S16rand-schurMass.right of south east),
      anchor=left of south west,
      name=levelVary-S16rand-schurBFBT40,
      convaxis, convaxisscale,
      title={\diagBFBT},
      ymajorticks=false,
    ]
      \addplot[convgraph=\colorB] table [x=iter,y=residual] {\dataB};
      \addplot[convgraph=\colorC] table [x=iter,y=residual] {\dataC};
      \addplot[convgraph=\colorD] table [x=iter,y=residual] {\dataD};
    \end{axis}

    \def\dataA{\datadir/T16rand_viscDrMagn8_schurBFBT23_bdrAmp1-1_k2_l4.txt}
    \def\dataB{\datadir/T16rand_viscDrMagn8_schurBFBT23_bdrAmp1-1_k2_l5.txt}
    \def\dataC{\datadir/T16rand_viscDrMagn8_schurBFBT23_bdrAmp1-1_k2_l6.txt}
    \def\dataD{\datadir/T16rand_viscDrMagn8_schurBFBT23_bdrAmp1-1_k2_l7.txt}
    \begin{axis}[
      at=(levelVary-S16rand-schurBFBT40.right of south east),
      anchor=left of south west,
      name=levelVary-S16rand-schurBFBT23,
      convaxis, convaxisscale,
      title={\wBFBT},
      ymajorticks=false,
    ]
      \addplot[convgraph=\colorB] table [x=iter,y=residual] {\dataB};
      \addplot[convgraph=\colorC] table [x=iter,y=residual] {\dataC};
      \addplot[convgraph=\colorD] table [x=iter,y=residual] {\dataD};
      \legend{$\ell=5$, $\ell=6$, $\ell=7$};
    \end{axis}

    \def\colorA{jr@purple}
    \def\colorB{jr@pink!75}
    \def\colorC{jr@orange}
    \def\colorD{jr@brown}
    \def\datadir{graphics/plotdata/comparison_orderVary}

    \def\dataA{\datadir/T16rand_viscDrMagn8_schurMass_k2_l5.txt}
    \def\dataB{\datadir/T16rand_viscDrMagn8_schurMass_k3_l5.txt}
    \def\dataC{\datadir/T16rand_viscDrMagn8_schurMass_k4_l5.txt}
    \def\dataD{\datadir/T16rand_viscDrMagn8_schurMass_k5_l5.txt}
    \begin{axis}[
      at=(levelVary-S16rand-schurMass.left of south west),
      anchor=left of north west,
      yshift=-0.6cm,
      name=orderVary-S16rand-schurMass,
      convaxis, convaxisscale,
      xlabel={GMRES iteration},
      ylabel={fixed level $\ell=5$ \\[4ex]
              $\norm{\text{res}}/\norm{\text{init res}}$},
    ]
      \addplot[convgraph=\colorA] table [x=iter,y=residual] {\dataA};
      \addplot[convgraph=\colorB] table [x=iter,y=residual] {\dataB};
      \addplot[convgraph=\colorC] table [x=iter,y=residual] {\dataC};
      \addplot[convgraph=\colorD] table [x=iter,y=residual] {\dataD};
    \end{axis}

    \def\dataA{\datadir/T16rand_viscDrMagn8_schurBFBT40_bdrAmp1-1_k2_l5.txt}
    \def\dataB{\datadir/T16rand_viscDrMagn8_schurBFBT40_bdrAmp1-1_k3_l5.txt}
    \def\dataC{\datadir/T16rand_viscDrMagn8_schurBFBT40_bdrAmp1-1_k4_l5.txt}
    \def\dataD{\datadir/T16rand_viscDrMagn8_schurBFBT40_bdrAmp1-1_k5_l5.txt}
    \begin{axis}[
      at=(orderVary-S16rand-schurMass.right of south east),
      anchor=left of south west,
      name=orderVary-S16rand-schurBFBT40,
      convaxis, convaxisscale,
      xlabel={GMRES iteration},
      ymajorticks=false,
    ]
      \addplot[convgraph=\colorA] table [x=iter,y=residual] {\dataA};
      \addplot[convgraph=\colorB] table [x=iter,y=residual] {\dataB};
      \addplot[convgraph=\colorC] table [x=iter,y=residual] {\dataC};
      \addplot[convgraph=\colorD] table [x=iter,y=residual] {\dataD};
    \end{axis}

    \def\dataA{\datadir/T16rand_viscDrMagn8_schurBFBT23_bdrAmp1-1_k2_l5.txt}
    \def\dataB{\datadir/T16rand_viscDrMagn8_schurBFBT23_bdrAmp1-1_k3_l5.txt}
    \def\dataC{\datadir/T16rand_viscDrMagn8_schurBFBT23_bdrAmp1-1_k4_l5.txt}
    \def\dataD{\datadir/T16rand_viscDrMagn8_schurBFBT23_bdrAmp1-1_k5_l5.txt}
    \begin{axis}[
      at=(orderVary-S16rand-schurBFBT40.right of south east),
      anchor=left of south west,
      name=orderVary-S16rand-schurBFBT23,
      convaxis, convaxisscale,
      xlabel={GMRES iteration},
      ymajorticks=false,
    ]
      \addplot[convgraph=\colorA] table [x=iter,y=residual] {\dataA};
      \addplot[convgraph=\colorB] table [x=iter,y=residual] {\dataB};
      \addplot[convgraph=\colorC] table [x=iter,y=residual] {\dataC};
      \addplot[convgraph=\colorD] table [x=iter,y=residual] {\dataD};
      \legend{$k=2$, $k=3$, $k=4$, $k=5$};
    \end{axis}
  \end{tikzpicture}
  \caption{%
    Comparison of Stokes solver convergence with
    \schurMass\ \infigure{(left column)},
    \diagBFBT\ \infigure{(middle column)}, and
    \wBFBT\ \infigure{(right column)} preconditioning.
    We fix the problem S16-rand, $\dynamicratio(\visc)=10^8$ while varying
    mesh refinement level $\ell$ \infigure{(top row)} and discretization order
    $k$ \infigure{(bottom row)}.
    This comparison shows that \wBFBT\ combines robust convergence of
    \diagBFBT\ with improved algorithmic scalability when $k$ is increased.
  }
  \label{fig:comparison}
\end{figure}
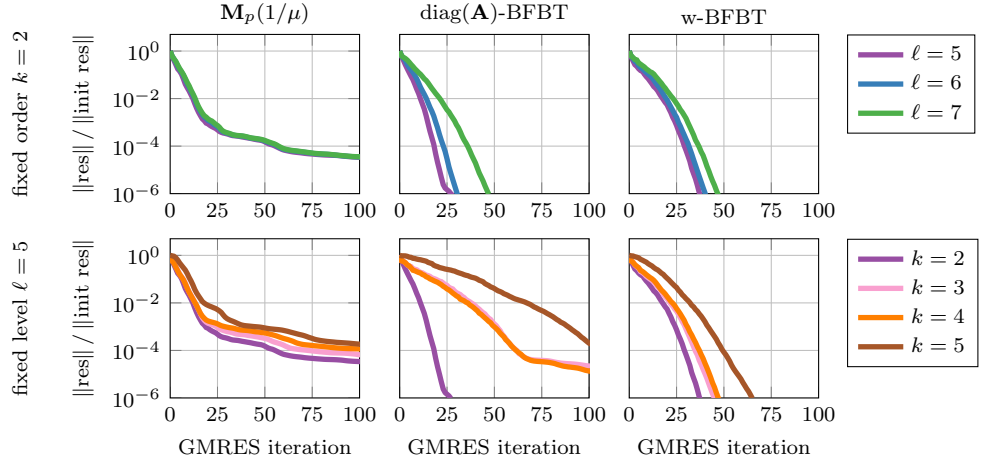

We compare convergence of the Stokes solver using the Schur complement
approximation \schurMass\ with \diagBFBT\ and with the proposed
\wBFBT.  The problem parameters are held fixed to S16-rand and
$\dynamicratio(\visc)=10^8$.  The numerical experiments are carried
out using different levels of mesh refinement $\ell=5,\ldots,7$ (for
fixed order $k=2$) and different discretization orders $k=2,\ldots,5$
(for fixed level $\ell=5$).  A level $\ell$ corresponds to a mesh of
$2^{3\ell}$ elements due to uniform refinement.
Note that for these tests, the applications of $\ViscStressMat^{-1}$,
$\left(\DivMat\WbfbtScalLeftMat^{-1}\GradMat\right)^{-1}$, and
$\left(\DivMat\WbfbtScalRightMat^{-1}\GradMat\right)^{-1}$
are approximated using a multigrid method (introduced in
\cref{sec:hmg}). These approximations are
sufficiently accurate, such that the comparison is indicative of the
effectiveness of the different Schur complement approximations. In
particular, improving the approximation does not
change the results, which
are presented in \cref{fig:comparison}.  In the
\infigure{left} two plots, the poor Schur complement approximation by
\schurMass\ for this problem setup can be observed clearly.
Convergence stagnates %
(similar results are found
in~\cite{MayBrownLePourhiet14,RudiMalossiIsaacEtAl15}).

Preconditioner \diagBFBT\ (\cref{fig:comparison}, \infigure{middle}) is able to
achieve fast convergence for discretization order $k=2$.
A limitation of \diagBFBT\ is a strong dependence on the order $k$.
This can be explained by the decreasing diagonal dominance in
the viscous block $\ViscStressMat$ with increasing order $k$: for higher
$k$ the approximation of $\ViscStressMat$ by $\diag(\ViscStressMat)$
deteriorates.
Note that numerical experiments with \massBFBT\ are not presented, because it
performs poorly in the presence of spatially-varying viscosities.  This leads
to the conclusion that the choice of the weighting matrices
$\BfbtScalLeftMat$, $\BfbtScalRightMat$ in $\SchurPcMat_{\BFBT}^{-1}$
crucially affects the quality of the Schur complement approximation.

The \wBFBT\ approximation delivers convergence that is nearly as fast
as in the \diagBFBT, $k=2$ case, but without the severe deterioration when
$k$ is increased (see \cref{fig:comparison}, \infigure{right}).  Thus, \wBFBT\
exhibits the robustness of \diagBFBT\ and additionally shows superior
algorithmic scalability with respect to~$k$.
Having illustrated the efficacy of \wBFBT\ for certain problem
parameters, we next establish spectral equivalence of
\wBFBT\ (\cref{sec:theory}) and then analyze in more
detail how crucial parameters influence convergence in
\cref{sec:robustness}.

\section{Spectral equivalence of \wBFBT}
\label{sec:theory}

Before we show spectral equivalence, we introduce notation and basic
definitions.

\subsection{Basic definitions}
\label{sec:theory-intro}

Let
  $\Omega\subset\R^d$
be a bounded domain with Lipschitz boundary
  $\partial\Omega$.
We denote as
  $\LL(\Omega)$
the class of real-valued square-integrable functions,
equipped with the usual $\LL$-inner product
$  \innerprod{u}{v}_{\LL(\Omega)}$
and induced norm
  $\norm{u}_{\LL(\Omega)}$, $u,v \in\LL(\Omega)$.
We also consider corresponding spaces of $d$-dimensional vector-valued
functions and
$(d \times d)$-dimensional tensor-valued functions with component-wise
multiplication, denoted by
  $\LL(\Omega)^{d}$
and
  $\LL(\Omega)^{d \times d}$.
The subspace of $\LL(\Omega)$ that does not contain constant functions
is denoted by
 $\FncSpPressMean$.
A bounded function, say $\alpha=\alpha(\bx)$, belongs to the space
  $\Lebesgue{\infty}(\Omega)$
by satisfying the following finite norm:
  $\norm{\alpha}_{\Lebesgue{\infty}(\Omega)} \coloneqq
    \esssup_{\bx\in\Omega} \abs{\alpha(\bx)} < \infty$.
We generalize the $\LL$-norms to classes of weighted $\LL_{\alpha}$-norms for
functions
  $f \in \LL(\Omega)^n$,
  $n \in \{1, d, d\times d\}$,
defined by
\begin{equation*}
  \norm{f}_{\LL_{\alpha}(\Omega)^n} \coloneqq
    \norm{\alpha^{\frac12} f}_{\LL(\Omega)^n}
  \quad\text{for }
    \alpha \in \Linf(\Omega) \text{, }
    0 < \alpha(\bx) \text{ a.e.\ in } \Omega.
\end{equation*}
Next, we introduce
  $\Sobolev{m}(\Omega)$,
with $m \ge 0$, which is the Sobolev space of $m$ derivatives in $\LL(\Omega)$,
and for $m=1$ we use the inner product
$
\innerprod{u}{v}_{\Sobolev{1}(\Omega)} \coloneqq
    \innerprod{u}{v}_{\LL(\Omega)} +
    \innerprod{\gradient u}{\gradient v}_{\LL(\Omega)},
$
inducing the norm
$\norm{u}_{\Sobolev{1}(\Omega)}$.
Functions in $\Sobolev{m}(\Omega)$ with vanishing
trace on the boundary $\partial\Omega$ belong to the space
  $\Sobolev{m}_{0}(\Omega)$.
Finally, we say that a function belongs to the class of
  $C^{\infty}(\Omega)$
if it has partial derivatives of any order in $\Omega$, and these
derivatives are continuous.

We transition from abstract definitions to fluid mechanics.
The differential operators acting on velocity
  $\vel\in\FncSpVel$
and pressure
  $\press\in\FncSpPress$
within the Stokes equations are defined in the sense of distributions:
\begin{equation*}
  \symgradient\vel \coloneqq \frac12 (\gradient\vel + \gradient\vel\transpose),
  \quad
  \DivOp\vel       \coloneqq -\divergence\vel,
  \quad
  \GradOp\press    \coloneqq \gradient\press.
\end{equation*}
Moreover, assume a sufficiently regular, bounded viscosity
  $\visc \in \Sobolev{1}(\Omega) \cap \Lebesgue{\infty}(\Omega)$
such that
  $0 < \viscmin \le \visc(\bx)$ a.e.\ in $\Omega$
and then define the viscous stress tensor
  $\viscstresstensor \coloneqq 2\visc\symgradient\vel$.
We denote the function space for velocity by
\begin{equation}
  \label{eq:fnc-sp-vel-bc}
  V \coloneqq
    \left\{
      \vel \in (\Sobolev{1}(\Omega))^{d} \mid
      \bn\cdot\vel = 0 \text{ on } \partial\Omega
    \right\},
\end{equation}
where $\bn\in\R^d$ is the outward unit normal vector at the boundary
$\partial\Omega$,
and the function space for pressure by
  $Q \coloneqq \FncSpPressMean$,
and we introduce the viscous stress operator with a heterogeneous viscosity
\begin{equation*}
  \ViscStressOp_{\visc} : V \rightarrow V\dual, \quad
  \ViscStressOp_{\visc}\vel \coloneqq
    -\divergence(2\visc\symgradient\vel) =
    -\divergence\viscstresstensor.
\end{equation*}
Given exterior forces acting on the fluid
  $\rhsvel \in V\dual$,
we consider the incompressible Stokes problem with free-slip and
no-normal flow
boundary conditions:
\begin{subequations}
\label{eq:abstr-stokes}
\begin{alignat}{2}
  \label{eq:abstr-momentum}
  \ViscStressOp_{\visc}\vel + \GradOp\press &= \rhsvel
    &&\quad\text{in }\Omega,\\
  \label{eq:abstr-mass}
  \DivOp\vel &= 0
    &&\quad\text{in }\Omega,\\
  \label{eq:abstr-stress-bc}
  \bT\left[ \viscstresstensor - \press\IdMat \right]\bn &= 0
    &&\quad\text{on }\partial\Omega,\\
  \label{eq:abstr-vel-bc}
  \vel\cdot\bn &= 0
    &&\quad\text{on }\partial\Omega,
\end{alignat}
\end{subequations}
in which we seek the velocity
  $\vel \in V$
and pressure
  $\press \in Q$.
On the boundary, we have outward unit normal vectors
  $\bn\in\R^{d}$
and tangential projectors
  $\bT \coloneqq \IdMat - \bn\bn\transpose$.
For the theoretical analysis sections, we altered the boundary conditions
compared to~\eqref{eq:dir-bc} to achieve a cleaner presentation.

For the definition of the \wBFBT\ approximation of the Schur complement, we
introduce a Poisson operator for higher regularity pressure functions
\begin{equation}
  \label{eq:press-poisson-op}
  \PressPoissonOp_{w}\adjoint : \FncSpPressHiReg\rightarrow\FncSpPress, \quad
  \PressPoissonOp_{w}\adjoint\press \coloneqq
    \DivOp w \GradOpExt\press,
\end{equation}
with an appropriate coefficient $w$ (see below) and augmented with homogeneous
Neumann boundary conditions,
  $\bn\cdot\GradOpExt\press = 0$.
The $\LL$-adjoint of $\PressPoissonOp_{w}\adjoint$ is denoted by
  $\PressPoissonOp_{w}$.
Finally, we define the \wBFBT\ approximation of the Schur complement
  $\SchurOp = \DivOp\ViscStressOp_{\visc}^{-1}\GradOp$
by:
\begin{equation}
  \label{eq:wbfbt-op}
  \SchurApproxOpWBFBT \coloneqq
    \PressPoissonOp_{\bfbtweightright}\adjoint
    (\DivOpExt\bfbtweightleft\,\ViscStressOp_{\visc}\,
     \bfbtweightright\GradOpExt)^{-1}
    \PressPoissonOp_{\bfbtweightleft},
\end{equation}
with sufficiently regular, bounded weight functions
  $\bfbtweightleft,\bfbtweightright \in
    \Sobolev{1}(\Omega) \cap \Lebesgue{\infty}(\Omega)$
such that
  $0 < \bfbtweight_{\mathrm{min}} \le
    \bfbtweightleft(\bx), \bfbtweightright(\bx)$
a.e.\ in $\Omega$.
Note that the definitions of the \wBFBT\ weight functions in the
discrete case~\eqref{eq:wbfbt} are reciprocal to definitions
in~\eqref{eq:press-poisson-op} and~\eqref{eq:wbfbt-op}, because in the discrete
case the weight functions were embedded into inverses of mass matrices.

\subsection{Main theorem on spectral equivalence of \wBFBT}
\label{sec:theory-thm}

One measure for the efficacy of a preconditioner consists of the ratio of the
maximal to minimal eigenvalues of the preconditioned system
  $\SchurPcMat_{\wBFBT}^{-1}\SchurMat$.  This section establishes
inequalities for spectral equivalence of
\wBFBT\ by providing bounds on that ratio.  The derivations are carried out in
an infinite-dimensional setting.  We begin by stating the main result of this
section in \cref{thm:main-result} and continue with proving this result using
a sequence of lemmas.

\begin{theorem}[Main result]
  \label{thm:main-result}
  Let
    $\hat{Q} = \FncSpPressMean\cap\Sobolev{1}(\Omega)$.
  If the left and right \wBFBT\ weight functions are equal to
  \begin{equation*}
    \bfbtweightleft(\bx) = \visc(\bx)^{-\frac12} = \bfbtweightright(\bx)
    \quad\text{for a.a. } \bx\in\Omega,
  \end{equation*}
  then the exact Schur complement is equivalent to the \wBFBT\ approximation
  such that
  \begin{equation*}
      \innerprod{\SchurApproxOpWBFBT\,\presstest}{\presstest}
    \le
      \innerprod{\SchurOp\presstest}{\presstest}
    \le
      C_{\wBFBT}
      \innerprod{\SchurApproxOpWBFBT\,\presstest}{\presstest}
    \quad\text{for all } \presstest\in\hat{Q},
  \end{equation*}
  where
  \begin{equation*}
    C_{\wBFBT} \coloneqq
      \Bigl(1 + \frac14 \norm{\gradient\visc}_{\Linf(\Omega)^{d}}^2\Bigr)
      \Bigl(C_{P,\visc}^2 + 1\Bigr)
      C_{K,\visc}^2
  \end{equation*}
  and the constants $C_{P,\visc},C_{K,\visc} > 0$ stem from weighted
  Poincar\'e--Friedrichs' and Korn's inequalities, respectively
  (see Remark~\ref{rmk:weighted-poincare-korn} for more information);
  the viscosity $\visc$ assumes the role of the weight function in the weighted
  inequalities.

  If the viscosity and the \wBFBT\ weight functions are constant,
  \begin{equation*}
    \visc \equiv 1, \quad
    \bfbtweightleft \equiv 1 \equiv \bfbtweightright,
  \end{equation*}
  then the exact Schur complement is equivalent to the \wBFBT\ approximation
  such that
  \begin{equation*}
      \innerprod{\SchurApproxOpWBFBT\,\presstest}{\presstest}
    \le
      \innerprod{\SchurOp\presstest}{\presstest}
    \le
      (C_P^2 + 1) \, C_K^2
      \innerprod{\SchurApproxOpWBFBT\,\presstest}{\presstest}
    \quad\text{for all } \presstest\in\hat{Q}
  \end{equation*}
  with constants $C_P,C_K > 0$ stemming from (the classical)
  Poincar\'e--Friedrichs' and Korn's inequalities, respectively.
\end{theorem}

\subsection{Proofs}
\label{sec:theory-proof}

The proof of \cref{thm:main-result} is established in the remainder of this
section.
In what follows, suprema are understood over spaces excluding operator kernels
that would cause a supremum to blow up.
The following basic, but hereafter frequently used, result is shown for
completeness of the discussion.

\begin{lemma}[sup-form of inverse operator]
  \label{lem:inv-op-sup-form}
  Let
  $V$ be a complete Hilbert space and
    $W \subseteq V$
  be a dense subspace.  Assume the linear operator
    $T : V \rightarrow V\dual$
  to be bounded, invertible, symmetric, and positive definite.
  Then for any
    $f \in V\dual$
  follows
  \begin{equation*}
    \innerprod{T^{-1} f}{f} =
      \sup_{w \in W} \frac{\innerprod{w}{f}^2}{\innerprod{Tw}{w}}.
  \end{equation*}
\end{lemma}
\begin{proof}
  Let $w \in W$, then with H\"older's inequality follows
  \begin{equation*}
      \innerprod{w}{f}^2
    =
      \innerprod{T^{1/2} w}{T^{-1/2} f}^2
    \le
      \norm{T^{1/2} w}^2 \norm{T^{-1/2} f}^2
    =
      \innerprod{T w}{w} \innerprod{T^{-1} f}{f}.
  \end{equation*}
  Additionally, let
    $v = T^{-1} f$
  and since $W \subseteq V$ dense, there exists a sequence
    $\{w_k\}_k \subset W$
  such that
    $w_k \rightarrow v = T^{-1} f \in V$, hence
  \begin{equation*}
      \frac{\innerprod{w_k}{f}^2}{\innerprod{Tw_k}{w_k}}
    \rightarrow
      \frac{\innerprod{v}{f}^2}{\innerprod{Tv}{v}}
    =
      \frac{\innerprod{T^{-1} f}{f}^2}{\innerprod{f}{T^{-1} f}}
    =
      \innerprod{T^{-1} f}{f},
  \end{equation*}
  which shows that equality is achieved in the limit.
\end{proof}

The next lemma establishes Schur complement properties that are
essential for deriving lower and upper
bounds in the spectral equivalence estimates.

\begin{lemma}[sup-form of Schur complement]
  \label{lem:schur-sup-form}
  With the definitions from \cref{sec:theory-intro},
  the following two equalities hold:
  \begin{equation}
    \label{eq:wbfbt-sup-form}
      \innerprod{\SchurApproxOpWBFBT\,\presstest}{\presstest}
    =
      \sup_{\press \in \hat{P}}
        \frac{ \innerprod{ \GradOpExt\press}
                         { \bfbtweightright\GradOpExt\presstest }^2 }
             { \innerprod{ \bfbtweightleft\ViscStressOp_{\visc}
                           \bfbtweightright\GradOpExt\press }
                         { \GradOpExt\press } }
    \quad\text{for all } \presstest \in \hat{Q},
  \end{equation}
  where
    $\hat{P} \coloneqq C^{\infty}(\Omega)$
  and
    $\hat{Q} \coloneqq \FncSpPressMean\cap\Sobolev{1}(\Omega)$,
  and
  \begin{equation}
    \label{eq:schur-sup-form}
      \innerprod{\SchurOp\presstest}{\presstest}
    =
      \sup_{\veltest \in V}
        \frac{ \innerprod{ \veltest }
                         { \bfbtweightright\GradOp\presstest }^2 }
             { \innerprod{ \bfbtweightleft\ViscStressOp_{\visc}
                           \bfbtweightright\veltest }
                         { \veltest } }
    \quad\text{for all } \presstest \in Q.
  \end{equation}
\end{lemma}
\begin{proof}
  For
    $\presstest \in \hat{Q}$,
  we use integration by parts on the left hand side
  of~\eqref{eq:wbfbt-sup-form} to obtain
  \begin{align*}
      \innerprod{\SchurApproxOpWBFBT\,\presstest}{\presstest}
    &=
      \int_{\Omega}
        \left( \bfbtweightright\GradOpExt
               (\DivOpExt\bfbtweightleft\ViscStressOp_{\visc}
                \bfbtweightright\GradOpExt)^{-1}
               \PressPoissonOp_{\bfbtweightleft}\presstest \right)
        \left(\GradOpExt\presstest\right) \intd\bx
      + b_1(\presstest)
    \\
    &=
      \int_{\Omega}
        \left( (\DivOpExt\bfbtweightleft\ViscStressOp_{\visc}
                \bfbtweightright\GradOpExt)^{-1}
               \PressPoissonOp_{\bfbtweightleft}\presstest \right)
        \left(\PressPoissonOp_{\bfbtweightright}\presstest\right) \intd\bx
      + b_1(\presstest) + b_2(\presstest)
  \end{align*}
  with boundary terms
  \begin{align*}
    b_1(\presstest) &\coloneqq
      -\int_{\partial\Omega}
        \left( \bn\cdot \bfbtweightright\GradOpExt
               (\DivOpExt\bfbtweightleft\ViscStressOp_{\visc}
                \bfbtweightright\GradOpExt)^{-1}
               \PressPoissonOp_{\bfbtweightleft}\presstest \right)
        \presstest \intd\bx,
    \\
    b_2(\presstest) &\coloneqq
      \int_{\partial\Omega}
        \left( (\DivOpExt\bfbtweightleft\ViscStressOp_{\visc}
                \bfbtweightright\GradOpExt)^{-1}
               \PressPoissonOp_{\bfbtweightleft}\presstest \right)
        \left(\bn\cdot \bfbtweightright\GradOpExt\presstest\right) \intd\bx.
  \end{align*}
  Using that $\hat{P} \subset \FncSpPressHiReg$ is
  dense, application of Lemma~\ref{lem:inv-op-sup-form} and further
  integration by parts yields
  \begin{align*}
      \innerprod{\SchurApproxOpWBFBT\,\presstest}{\presstest}
    &=
      \sup_{\press \in \hat{P}}
        \frac{ \innerprod{ \press}
                         { \PressPoissonOp_{\bfbtweightright}\presstest }^2 }
             { \innerprod{ \DivOpExt\bfbtweightleft\ViscStressOp_{\visc}
                           \bfbtweightright\GradOpExt\press }
                         { \press } }
      + b_1(\presstest) + b_2(\presstest)
    \\
    &=
      \sup_{\press \in \hat{P}}
        \frac{ \left( \innerprod{ \GradOpExt\press}
                                { \bfbtweightright\GradOpExt\presstest }
                      + b_3(\press,\presstest) \right)^2 }
             { \innerprod{ \bfbtweightleft\ViscStressOp_{\visc}\bfbtweightright
                           \GradOpExt\press }
                         { \GradOpExt\press } + b_4(\press) }
      + b_1(\presstest) + b_2(\presstest)
  \end{align*}
  with boundary terms
  \begin{align*}
    b_3(\press,\presstest) &\coloneqq
      -\int_{\partial\Omega}
        \press \left(\bn\cdot\bfbtweightright\GradOpExt\presstest\right)
        \intd\bx,
    \\
    b_4(\press) &\coloneqq
      -\int_{\partial\Omega}
        \left( \bn\cdot \bfbtweightleft\ViscStressOp_{\visc}\bfbtweightright
               \GradOpExt\press \right)
        \press \intd\bx.
  \end{align*}
  Because the operator
    $\PressPoissonOp_{\bfbtweightright}\adjoint$
  from~\eqref{eq:press-poisson-op} is augmented with homogeneous Neumann
  boundary conditions,
    $\bn\cdot\GradOpExt\press = 0$,
  the boundary terms
    $b_1(\presstest)$,
    $b_2(\presstest)$, and
    $b_3(\press,\presstest)$
  vanish.
  In addition, $\press \in \hat{P}$ is sufficiently regular for the term
    $b_4(\press)$
  to be well-defined and it equals to zero because the velocity $\vel$
  satisfies
    $\bn\cdot\vel=0$ on $\partial\Omega$.
  Hence,~\eqref{eq:wbfbt-sup-form} follows.

  To show~\eqref{eq:schur-sup-form}, let
    $\presstest \in Q$,
  then for the exact Schur complement we apply integration by parts with a
  vanishing boundary term
  \begin{equation*}
      \innerprod{\SchurOp\presstest}{\presstest}
    =
      \innerprod{\DivOp\ViscStressOp_{\visc}^{-1}\GradOp\presstest}{\presstest}
    =
      \innerprod{\ViscStressOp_{\visc}^{-1}\GradOp\presstest}{\GradOp\presstest}
    =
      \innerprod{ \bfbtweightright^{-1} \ViscStressOp_{\visc}^{-1}
                  \bfbtweightleft^{-1} \bfbtweightleft\GradOp\presstest }
                { \bfbtweightright\GradOp\presstest }
  \end{equation*}
  and~\eqref{eq:schur-sup-form} follows from Lemma~\ref{lem:inv-op-sup-form}.
\end{proof}

A direct consequence of Lemma~\ref{lem:schur-sup-form} is the following lower
bound.

\begin{corollary}[Lower bound, $\SchurApproxOpWBFBT\lesssim\SchurOp$]
  \label{cor:wbfbt-lower-bound}
  The exact Schur complement is bounded by the \wBFBT\ approximation
  from below, i.e.,
  \begin{equation}
    \label{eq:wbfbt-lower-bound}
      \innerprod{\SchurApproxOpWBFBT\,\presstest}{\presstest}
    \le
      \innerprod{\SchurOp\presstest}{\presstest}
    \quad\text{for all } \presstest\in\hat{Q},
  \end{equation}
  where
    $\hat{Q} \coloneqq \FncSpPressMean\cap\Sobolev{1}(\Omega)$.
\end{corollary}
\begin{proof}
  Let
    $\hat{P} = C^{\infty}(\Omega)$
  and
    $\presstest \in \hat{Q}$.
  Since $\GradOpExt$ maps $\hat{P}$ into $V$, we
  combine~\eqref{eq:wbfbt-sup-form} and~\eqref{eq:schur-sup-form} to get
  \begin{equation*}
      \innerprod{\SchurApproxOpWBFBT\,\presstest}{\presstest}
    =
      \sup_{\press \in \hat{P}}
        \frac{ \innerprod{ \GradOpExt\press}
                         { \bfbtweightright\GradOp\presstest }^2 }
             { \innerprod{ \bfbtweightleft\ViscStressOp_{\visc}\bfbtweightright
                           \GradOpExt\press }
                         { \GradOpExt\press } }
    \le
      \sup_{\veltest \in V}
        \frac{ \innerprod{ \veltest }{ \bfbtweightright\GradOp\presstest }^2 }
             { \innerprod{ \bfbtweightleft\ViscStressOp_{\visc}\bfbtweightright
                           \veltest }{ \veltest } }
    =
      \innerprod{\SchurOp\presstest}{\presstest},
  \end{equation*}
  and obtain the result~\eqref{eq:wbfbt-lower-bound}.
\end{proof}

We begin the derivation of an upper bound for the case of constant
viscosity $\visc\equiv 1$.
Note that $\SchurApproxOpWBFBT$ is scaling invariant with respect to constants
multiplied to the \wBFBT\ weight functions $\bfbtweightleft$,
$\bfbtweightright$.  Hence, it always assumes the correct scaling of $\SchurOp$
independent of the viscosity constant.
The
result for constant viscosity presented below in
Lemma~\ref{lem:bfbt-upper-bound} is generalized to variable viscosity in
Lemma~\ref{lem:wbfbt-upper-bound}.
While Lemma~\ref{lem:bfbt-upper-bound} is a special case of
Lemma~\ref{lem:wbfbt-upper-bound}, we first prove the result for constant
viscosity as the arguments are less technical and easier to follow.  In the
proof of the result for variable viscosity, we build on some of the arguments
from the constant viscosity case and thus avoid unnecessary duplication.

\begin{lemma}[Upper bound, $\SchurOp\lesssim\SchurApproxOpWBFBT$, for
              constant $\visc$]
  \label{lem:bfbt-upper-bound}
  Assume a constant viscosity
    $\visc \equiv 1$
  and constant \wBFBT\ weight functions
    $\bfbtweightleft \equiv 1 \equiv \bfbtweightright$,
  and, as before in Lemma~\ref{lem:schur-sup-form}, let
    $\hat{Q} = \FncSpPressMean\cap\Sobolev{1}(\Omega)$.
  Then the exact Schur complement is bounded by the \wBFBT\ approximation from
  above by
  \begin{equation}
    \label{eq:bfbt-upper-bound}
      \innerprod{\SchurOp\presstest}{\presstest}
    \le
      (C_P^2 + 1) \, C_K^2
      \innerprod{\SchurApproxOpWBFBT\,\presstest}{\presstest}
    \quad\text{for all } \presstest\in\hat{Q}
  \end{equation}
  with constants $C_P,C_K > 0$ stemming from Poincar\'e--Friedrichs' and Korn's
  inequalities, respectively.
\end{lemma}
\begin{proof}
  Let
    $\hat{P} = C^{\infty}(\Omega)$
  and
    $\presstest \in \hat{Q}$,
  then due to~\eqref{eq:wbfbt-sup-form} we can write
  \begin{equation}
    \label{eq:bfbt-upper-bound:ansatz}
      \innerprod{\SchurApproxOpWBFBT\,\presstest}{\presstest}
    =
      \sup_{\press \in \hat{P}}
        \frac{ \innerprod{\GradOpExt\press}{\GradOpExt\presstest}^2 }
             { \norm{\GradOpExt\press}_{\FncSpVel}^2 } \,
        \frac{ \norm{\GradOpExt\press}_{\FncSpVel}^2 }
             { \innerprod{ \ViscStressOp_{1}\GradOpExt\press }
                         { \GradOpExt\press } }.
  \end{equation}
  To estimate the second factor on the right-hand side
  of~\eqref{eq:bfbt-upper-bound:ansatz},
  note that
  \begin{equation*}
      \innerprod{\ViscStressOp_{1}\GradOpExt\press}{\GradOpExt\press}
    =
      2 \innerprod{\symgradient\GradOpExt\press}
                  {\symgradient\GradOpExt\press}
    =
      2 \innerprod{\gradient\GradOpExt\press}
                  {\gradient\GradOpExt\press}
    =
      2 \norm{\gradient\GradOpExt\press}_{\FncSpVelGrad}^2,
  \end{equation*}
  where we used that
  $\bn\cdot\GradOpExt\press = 0$
  on the boundary and that $\gradient\GradOpExt\press$ is symmetric. Thus,
  \begin{equation}
    \label{eq:bfbt-upper-bound:2nd-factor}
      \innerprod{\ViscStressOp_{1}\GradOpExt\press}{\GradOpExt\press}
    \le
      2 \norm{\GradOpExt\press}_{\FncSpVel}^2.
  \end{equation}
  For the first factor on the right-hand side
  of~\eqref{eq:bfbt-upper-bound:ansatz}, observe that for any
    $\veltest \in V$
  there exists a sequence
    $\{\press_i\}_i \subset \hat{P} \subset \FncSpPressHiReg$
  such that
    $\PressPoissonOp_{1}\adjoint\press_i = \DivOpExt\GradOpExt\press_i
      \rightarrow \DivOpExt\veltest$,
  since $\PressPoissonOp_{1}\adjoint$ is invertible, where convergence is with
  respect to the $\LL$-norm.  Thus,
  \begin{equation}
    \label{eq:bfbt-upper-bound:1st-factor}
      \sup_{\press \in \hat{P}}
        \frac{ \innerprod{\GradOpExt\press}{\GradOpExt\presstest} }
             { \norm{\GradOpExt\press}_{\FncSpVel} }
    =
      \sup_{\veltest \in V}
        \frac{ \innerprod{\veltest}{\GradOpExt\presstest} }
             { \norm{\veltest}_{\FncSpVel} }
    =
      \norm{\GradOpExt\presstest}_{\FncSpVelDual}.
  \end{equation}
  Combining~\eqref{eq:bfbt-upper-bound:2nd-factor}
  and~\eqref{eq:bfbt-upper-bound:1st-factor} provides the following estimate
  for the \wBFBT\ Schur complement
  approximation~\eqref{eq:bfbt-upper-bound:ansatz}:
  \begin{equation}
    \label{eq:bfbt-upper-bound:bfbt-norm-bound}
      \innerprod{\SchurApproxOpWBFBT\,\presstest}{\presstest}
    \ge
      \frac{1}{2} \norm{\GradOpExt\presstest}_{\FncSpVelDual}^2.
  \end{equation}
  The exact Schur complement, on the other hand, in the
  form~\eqref{eq:schur-sup-form} from Lemma~\ref{lem:schur-sup-form}, can be
  bounded by
  \begin{equation}
    \label{eq:bfbt-upper-bound:schur-exact}
      \innerprod{\SchurOp\presstest}{\presstest}
    =
      \sup_{\veltest \in V}
        \frac{ \innerprod{\veltest}{\GradOpExt\presstest}^2 }
             { \innerprod{\ViscStressOp_{1}\veltest}{\veltest} }
    \le
      \sup_{\veltest \in V}
        \frac{ \norm{\veltest}_{\FncSpVel}^2
               \norm{\GradOpExt\presstest}_{\FncSpVelDual}^2 }
             { 2 \norm{\symgradient\veltest}_{\FncSpVelGrad}^2 }.
  \end{equation}
  With Poincar\'e--Friedrichs' inequality,
  \begin{equation*}
      \norm{\veltest}_{\LL(\Omega)^d}
    \le
      C_P \norm{\gradient\veltest}_{\FncSpVelGrad},
  \end{equation*}
  where the constant
    $C_P > 0$
  depends on the domain $\Omega$, and Korn's inequality,
  \begin{equation*}
      \norm{\gradient\veltest}_{\FncSpVelGrad}
    \le
      C_K \norm{\symgradient\veltest}_{\FncSpVelGrad},
  \end{equation*}
  with constant
    $C_K > 0$,
  we obtain
  \begin{equation*}
      \frac{1}{(C_P^2 + 1)} \norm{\veltest}_{\FncSpVel}^2
    \le
      \norm{\gradient\veltest}_{\FncSpVelGrad}^2
    \le
      C_K^2 \norm{\symgradient\veltest}_{\FncSpVelGrad}^2,
  \end{equation*}
  and substituting this into~\eqref{eq:bfbt-upper-bound:schur-exact} gives
  \begin{equation}
    \label{eq:bfbt-upper-bound:schur-norm-bound}
      \innerprod{\SchurOp\presstest}{\presstest}
    \le
      \frac{(C_P^2 + 1) \, C_K^2}{2}
      \norm{\GradOpExt\presstest}_{\FncSpVelDual}^2.
  \end{equation}
  Together
  with~\eqref{eq:bfbt-upper-bound:bfbt-norm-bound}, this yields the desired
  result~\eqref{eq:bfbt-upper-bound}.
\end{proof}

We complete the presentation of spectral equivalence by deriving an upper bound
for problems with variable viscosities.

\begin{lemma}[Upper bound, $\SchurOp\lesssim\SchurApproxOpWBFBT$, for
              variable $\visc$]
  \label{lem:wbfbt-upper-bound}
  As before in Lemma~\ref{lem:schur-sup-form}, let
    $\hat{Q} = \FncSpPressMean\cap\Sobolev{1}(\Omega)$.
  If the left and right \wBFBT\ weight functions are equal to
  \begin{equation}
    \label{eq:wbfbt-w-choice}
    \bfbtweightleft(\bx) = \visc(\bx)^{-\frac12} = \bfbtweightright(\bx)
    \quad\text{for a.a. } \bx\in\Omega,
  \end{equation}
  then the exact Schur complement is bounded by the \wBFBT\ approximation from
  above by
  \begin{equation}
    \label{eq:wbfbt-upper-bound}
      \innerprod{\SchurOp\presstest}{\presstest}
    \le
      \Bigl(1 + \frac14 \norm{\gradient\visc}_{\Linf(\Omega)^{d}}^2\Bigr)
      \Bigl(C_{P,\visc}^2 + 1\Bigr)
      C_{K,\visc}^2
      \innerprod{\SchurApproxOpWBFBT\,\presstest}{\presstest}
    \quad\text{for all } \presstest\in\hat{Q}
  \end{equation}
  with constants $C_{P,\visc},C_{K,\visc} > 0$ stemming from weighted
  Poincar\'e--Friedrichs' and Korn's inequalities, respectively, where $\visc$
  assumes the role of the weight function.
\end{lemma}
\begin{proof}
  Let the weight functions be equal,
    $\bfbtweightleft\equiv\bfbtweight\equiv\bfbtweightright$,
  but (for now) otherwise arbitrary subject to the condition
    $0 < \bfbtweight_{\mathrm{min}} \le \bfbtweight(\bx)$
  for a.a.\ $\bx\in\Omega$. At the end of this proof we will argue the
  special role of the choice~\eqref{eq:wbfbt-w-choice} for the weight functions.
  In addition, let
    $\hat{P} = C^{\infty}(\Omega)$
  and
    $\presstest \in \hat{Q}$,
  then due to~\eqref{eq:wbfbt-sup-form} we can write
  \begin{equation}
    \label{eq:wbfbt-upper-bound:ansatz}
      \innerprod{\SchurApproxOpWBFBT\,\presstest}{\presstest}
    =
      \sup_{\press \in \hat{P}}
        \frac{ \innerprod{\GradOpExt\press}
                         {\bfbtweight\GradOpExt\presstest}^2 }
             { \norm{\GradOpExt\press}_{\FncSpVel}^2 } \,
        \frac{ \norm{\GradOpExt\press}_{\FncSpVel}^2 }
             { \innerprod{ \bfbtweight\ViscStressOp_{\visc}
                           \bfbtweight\GradOpExt\press }
                         { \GradOpExt\press } }.
  \end{equation}
  We begin by estimating the second factor
  on the right-hand side of~\eqref{eq:wbfbt-upper-bound:ansatz}.
  For an arbitrary
    $\veltest\in\FncSpVel$,
  observe that
  $ %
    \symgradient(\bfbtweight\veltest) =
      \bfbtweight\symgradient\veltest + \gradient\bfbtweight\otimes\veltest,
  $ %
  where ``$\otimes$'' denotes the outer product of two vectors in $\R^d$,
  and thus
  \begin{equation*}
      \innerprod{\bfbtweight\ViscStressOp_{\visc}\bfbtweight\veltest}{\veltest}
    =
      2 \innerprod{ \visc
                    \symgradient(\bfbtweight\veltest) }
                  { \symgradient(\bfbtweight\veltest) }
    =
      2 \norm{ \sqrt{\visc}\bfbtweight\symgradient\veltest +
               \sqrt{\visc}\gradient\bfbtweight\otimes\veltest
             }_{\FncSpVelGrad}^2.
  \end{equation*}
  Applying the triangle inequality and then H\"older's inequality to
  the resulting terms,
  \begin{align*}
      \norm{\sqrt{\visc}\bfbtweight\symgradient\veltest}_{\FncSpVelGrad}^2
    &\le
      \norm{\sqrt{\visc}\bfbtweight}_{\Linf(\Omega)}^2
      \norm{\symgradient\veltest}_{\FncSpVelGrad}^2,
    \\
      \norm{\sqrt{\visc}\gradient\bfbtweight\otimes\veltest}_{\FncSpVelGrad}^2
    &\le
      \norm{\sqrt{\visc}\gradient\bfbtweight}_{\Linf(\Omega)^{d}}^2
      \norm{\veltest}_{\LL(\Omega)^{d}}^2,
  \end{align*}
  and thus we obtain the estimate
  \begin{equation*}
      \innerprod{\bfbtweight\ViscStressOp_{\visc}\bfbtweight\veltest}{\veltest}
    \le
      2 C_{\ViscStressOp,\visc,\bfbtweight} \norm{\veltest}_{\FncSpVel}^2,
  \end{equation*}
  where
  \begin{equation}
    \label{eq:wbfbt-upper-bound:wbfbt-norm-bound-const}
    C_{\ViscStressOp,\visc,\bfbtweight} \coloneqq
      \norm{\sqrt{\visc}\bfbtweight}_{\Linf(\Omega)}^2 +
      \norm{\sqrt{\visc}\gradient\bfbtweight}_{\Linf(\Omega)^{d}}^2.
  \end{equation}
  Similarly to \eqref{eq:bfbt-upper-bound:1st-factor}
  and \eqref{eq:bfbt-upper-bound:bfbt-norm-bound}, we obtain
  the following estimate for the \wBFBT\ Schur complement
  approximation~\eqref{eq:wbfbt-upper-bound:ansatz}:
  \begin{equation}
    \label{eq:wbfbt-upper-bound:wbfbt-norm-bound}
      \innerprod{\SchurApproxOpWBFBT\,\presstest}{\presstest}
    \ge
      \frac{1}{2 C_{\ViscStressOp,\visc,\bfbtweight}}
      \norm{\bfbtweight\GradOpExt\presstest}_{\FncSpVelDual}^2.
  \end{equation}
  Proceeding with the exact Schur complement, we obtain
  from~\eqref{eq:schur-sup-form} in Lemma~\ref{lem:schur-sup-form} that
  \begin{equation}
    \label{eq:wbfbt-upper-bound:schur-exact}
      \innerprod{\SchurOp\presstest}{\presstest}
    =
      \sup_{\veltest \in V}
        \frac{ \innerprod{\bfbtweight^{-1}\veltest}
                         {\bfbtweight\GradOp\presstest}^2 }
             { \innerprod{\ViscStressOp_{\visc}\veltest}{\veltest} }
    \le
      \sup_{\veltest \in V}
        \frac{ \norm{\bfbtweight^{-1}\veltest}_{\FncSpVel}^2
               \norm{\bfbtweight\GradOp\presstest}_{\FncSpVelDual}^2 }
             { 2 \norm{\sqrt{\visc}\symgradient\veltest}_{\FncSpVelGrad}^2 }.
  \end{equation}
  We require a weighted Poincar\'e--Friedrichs' inequality
  (see Remark~\ref{rmk:weighted-poincare-korn} for details),
  \begin{equation}
    \label{eq:wbfbt-upper-bound:w-poincare}
      \norm{\veltest}_{\LL_{\bfbtweight^{-2}}(\Omega)^d}
    \le
      C_{P,\bfbtweight^{-2}}
      \norm{\gradient\veltest}_
        {\LL_{\bfbtweight^{-2}}(\Omega)^{d \times d}},
  \end{equation}
  and also a weighted Korn's inequality
  (see Remark~\ref{rmk:weighted-poincare-korn} for more information),
  \begin{equation}
    \label{eq:wbfbt-upper-bound:w-korn}
      \norm{\gradient\veltest}_{\LL_{\visc}(\Omega)^{d \times d}}
    \le
      C_{K,\visc}
      \norm{\symgradient\veltest}_{\LL_{\visc}(\Omega)^{d \times d}}.
  \end{equation}
  With~\eqref{eq:wbfbt-upper-bound:w-poincare}
  and~\eqref{eq:wbfbt-upper-bound:w-korn}, we are able to
  bound~\eqref{eq:wbfbt-upper-bound:schur-exact} from above:
  \begin{equation}
    \label{eq:wbfbt-upper-bound:schur-norm-bound}
      \innerprod{\SchurOp\presstest}{\presstest}
    \le
      \frac{\Bigl(C_{P,\bfbtweight^{-2}}^2 + 1\Bigr) C_{K,\visc}^2}{2}
      \left(
        \sup_{\veltest \in V}
          \frac{ \norm{\gradient\veltest}_
                   {\LL_{\bfbtweight^{-2}}(\Omega)^{d \times d}} }
               { \norm{\gradient\veltest}_
                   {(\LL_{\visc}(\Omega))^{d \times d}} }
      \right)
      \norm{\bfbtweight\GradOp\presstest}_{\FncSpVelDual}^2.
  \end{equation}

  The supremum term in~\eqref{eq:wbfbt-upper-bound:schur-norm-bound}
  and the constant $C_{\ViscStressOp,\visc,\bfbtweight}$
  in~\eqref{eq:wbfbt-upper-bound:wbfbt-norm-bound-const} motivate the choice
  for the weight $\bfbtweight$ to be
  \begin{equation*}
    \bfbtweight \coloneqq \visc^{-\frac12}.
  \end{equation*}
  Then the supremum in~\eqref{eq:wbfbt-upper-bound:schur-norm-bound} vanishes
  and~\eqref{eq:wbfbt-upper-bound:wbfbt-norm-bound-const} simplifies to
  \begin{equation*}
    C_{\ViscStressOp,\visc,\bfbtweight} =
      1 + \frac14 \norm{\gradient\visc}_{\Linf(\Omega)^{d}}^2.
  \end{equation*}
  Substituting this into~\eqref{eq:wbfbt-upper-bound:wbfbt-norm-bound}
  together with inequality~\eqref{eq:wbfbt-upper-bound:schur-norm-bound}
  yields the desired result~\eqref{eq:wbfbt-upper-bound}.
\end{proof}

\begin{remark}
  \label{rmk:weighted-poincare-korn}
  In the proof of Lemma~\ref{lem:wbfbt-upper-bound} we utilized a weighted
  Poincar\'e--Friedrichs' inequality, for which the optimal constant is
  \begin{equation*}
    C_{P,\visc} =
      \sup_{\veltest \in V}
        \frac{ \norm{\veltest}_{\LL_{\visc}(\Omega)^d} }
             { \norm{\gradient\veltest}_{\LL_{\visc}(\Omega)^{d \times d}} },
  \end{equation*}
  where the viscosity takes the role of the weight function.
  While weighted Poincar\'e and Friedrichs' inequalities have been investigated
  in the literature numerous times, usually they are proven by contradiction
  and scaling arguments, which does not provide information about the
  constants.  If explicit constants are found, they depend, in general, on the
  weight such that the resulting estimates are too pessimistic, e.g.,
    $C_{P,\visc} = \LandauO(\dynamicratio(\visc))$.
  Knowledge of constants that are robust with respect to weight functions is
  limited.
  In the context of a posteriori error estimates for finite elements,
  weight-independent constants could be found for convex domains and
  weights that are
  a positive power of a non-negative concave function \cite{ChuaWheeden06}.
  These results were refined for star-shaped domains under certain assumptions
  for the weights \cite{Verfuerth13}.
  For another class of weights, namely quasi-monotone piecewise constant weight
  functions, robust constants were derived in \cite{PechsteinScheichl13}.

  In addition to weighted Poincar\'e--Friedrichs', we utilized a weighted
  Korn's inequality in the proof of Lemma~\ref{lem:wbfbt-upper-bound}.  The
  optimal constant for this inequality is
  \begin{equation*}
    C_{K,\visc} =
      \sup_{\veltest \in V}
        \frac{ \norm{\gradient\veltest}_{\LL_{\visc}(\Omega)^{d \times d}} }
             { \norm{\symgradient\veltest}_
                    {\LL_{\visc}(\Omega)^{d \times d}} }.
  \end{equation*}
  As for $C_{P,\visc}$, straightforward estimation results in an overly
  pessimistic weight-dependent constant, namely
    $C_{K,\visc} = \LandauO(\dynamicratio(\visc))$,
  \cite{JohnKaiserNovo16}.
  Other work utilizing weighted Korn's inequalities usually aims to
  derive inequalities for special domain shapes, e.g.,
  \cite{AcostaDuranLombardi06}.

  In summary, robust constants for weighted Poincar\'e--Friedrichs' and
  Korn's inequalities for general weight functions are difficult to obtain and
  limitations exist in the form of assumptions on the weights.
  Further work on this topic could improve the constants for the spectral
  equivalence of \wBFBT\ but is beyond the scope of this paper.
\end{remark}

\section{Robustness of \wBFBT}
\label{sec:robustness}

In this section, we analyze the robustness properties of the widely used Schur
complement approximation \schurMass\
and the new \wBFBT\ via numerical
experiments.  Furthermore, we calculate the spectra for both approaches and
thus support the discussion in \cref{sec:theory} about theoretical eigenvalue
bounds with numerical results.
The comparison of \schurMass\ and \wBFBT\ is of particular importance,
because of the widespread use of the inverse viscosity-weighted mass matrix.
It is therefore of interest to determine when convergence with \schurMass\
deteriorates and using \wBFBT\ becomes beneficial.
A comparison with \diagBFBT\ was not performed because \cref{sec:comparison}
already showed that \diagBFBT\ performs similarly or worse than \wBFBT,
hence there are no advantages in using \diagBFBT\ over \wBFBT.

For the numerical experiments in this section, we return to the definitions and
setup from \cref{sec:benchmark}.
To apply the inverse of \schurMass, we diagonalize the mass matrix of the
discontinuous, modal pressure space by forming its lumped version
\eqref{eq:lumping}.
Moreover, to apply the approximate inverse of the viscous block
in~\eqref{eq:precond-stokes} we use the same multigrid method for each
of the two Schur approximations; this multigrid method is also used
for the inverse operators of \wBFBT\ in~\eqref{eq:wbfbt}.  The details
of the multigrid method are provided in \cref{sec:hmg}.
To compare the robustness, we vary two problem
parameters:
(i) the number of randomly placed sinkers $n$ and
(ii) the dynamic ratio $\dynamicratio(\visc)$.
The parameter $n$ influences the geometric complexity of the viscosity
$\visc$ while $\dynamicratio(\visc)$ controls the magnitude of
viscosity gradients.

\cref{tab:robustness-schurMass,tab:robustness-wBFBT} present the number of
GMRES iterations for a $10^{-6}$ residual reduction in the Euclidean norm.
Observe that for the S1-rand problem, the iteration count is essentially the
same for both \schurMass\ and \wBFBT, and that it stays stable across all
dynamic ratios
  $\dynamicratio(\visc) = 10^4,\ldots,10^{10}$.
Hence for this simple problem, \wBFBT\ has no advantages and its
additional computational cost makes it less efficient than \schurMass.
However,
the limitations of the \schurMass\ approach become apparent by increasing the
number of randomly positioned sinkers.  Two observations for \schurMass\ can be
made from \cref{tab:robustness-schurMass}.
First, the number of GMRES iterations rises with increasing number
of sinkers (factor $\sim$80 increase for $n=1,\ldots,28$,
$\dynamicratio(\visc)=10^8$).
Second, in a multi-sinker setup the dependence on $\dynamicratio(\visc)$
becomes more severe (factor $\sim$50 increase for $n=24$,
$\dynamicratio(\visc)=10^4,\ldots,10^{10}$).
This demonstrates that \schurMass\ is a poor
approximation of the Schur complement for certain classes of
problems with highly heterogeneous viscosities.

The advantages in robustness of the \wBFBT\ preconditioner are demonstrated in
\cref{tab:robustness-wBFBT}.  Compared to \schurMass, the number of GMRES
iterations is stable and the increase over the whole range of problem
parameters is just a factor of 2.  Only 60 iterations are needed for
the most extreme problem, namely S28-rand, $\dynamicratio(\visc)=10^{10}$,
for which convergence with \schurMass\ essentially stagnated.

\begin{table}
  \caption[Robustness classification for Schur complement approximations
           \schurMass\ and \wBFBT.]{%
    Robustness classification for Schur complement approximations
    \subref{tab:robustness-schurMass}~\schurMass\
    and
    \subref{tab:robustness-wBFBT}~\wBFBT\
    in terms of number of GMRES iterations ($10^{-6}$ residual
    reduction, GMRES restart every $100$ iterations).
    Number of randomly placed sinkers (\#sinkers) is increased across
    rows, while dynamic ratio ($\dynamicratio(\visc)$) is increased
    across columns.
    Discretization is fixed at $k=2$, $\ell=7$.
  }
  \label{tab:robustness}
  \begin{subtable}[b]{0.48\columnwidth}
    \caption{\schurMass}
    \label{tab:robustness-schurMass}
    \centering
    \footnotesize
    \setlength{\tabcolsep}{0.5em}  %
    \begin{tabular}{lrrrr}
      \toprule
      $\#\text{sinkers} \setminus \dynamicratio(\visc)$ &
      $10^4$ & $10^6$ & $10^8$ & $10^{10}$ \\
      \midrule
      \phantom{1}S1-rand &  29 &  31 &   31 &       29 \\
      \phantom{1}S4-rand &  53 &  63 &   71 &       80 \\
      \phantom{1}S8-rand &  64 &  79 &   93 &      165 \\
                S12-rand &  70 &  86 &   99 &      180 \\
                S16-rand &  85 & 167 &  231 &      891 \\
                S20-rand &  84 & 167 &  380 &      724 \\
                S24-rand & 117 & 286 & 3279 &     5983 \\
                S28-rand & 108 & 499 & 2472 & $>$10000 \\
      \bottomrule
    \end{tabular}
  \end{subtable}
  \hfill
  \begin{subtable}[b]{0.48\columnwidth}
    \caption{\wBFBT}
    \label{tab:robustness-wBFBT}
    \centering
    \footnotesize
    \setlength{\tabcolsep}{0.5em}  %
    \begin{tabular}{lrrrr}
      \toprule
      $\#\text{sinkers} \setminus \dynamicratio(\visc)$ &
      $10^4$ & $10^6$ & $10^8$ & $10^{10}$ \\
      \midrule
      \phantom{1}S1-rand & 29 & 29 & 29 & 30 \\
      \phantom{1}S4-rand & 39 & 41 & 42 & 44 \\
      \phantom{1}S8-rand & 38 & 40 & 41 & 44 \\
                S12-rand & 38 & 40 & 43 & 45 \\
                S16-rand & 40 & 45 & 47 & 48 \\
                S20-rand & 34 & 36 & 37 & 38 \\
                S24-rand & 31 & 32 & 39 & 55 \\
                S28-rand & 29 & 31 & 42 & 60 \\
      \bottomrule
    \end{tabular}
  \end{subtable}
\end{table}

More insight concerning the different convergence behaviors can be gained from
the eigenvalues in \cref{fig:robustness-eigvals}.
The plots in that figure are for two-dimensional multi-sinker problems,
which are analogous to the three-dimensional benchmark problems from
\cref{sec:benchmark}.
We discretize the problems on triangular meshes utilizing the FEniCS
library~\cite{LoggMardalGarth12}. We choose \discrPbubblePdisc{2}{1} finite
elements~\cite{CrouzeixRaviart73, DoneaHuerta03} because they represent a close
analog to the \discrQPdisc{2}{1} elements, which are employed on
(three-dimensional) hexahedral meshes.
Each plot shows the eigenvalues of the exact Schur complement
and the eigenvalues of the preconditioned Schur complement for the \schurMass\
and \wBFBT\ approximations,
where all inverse matrices, e.g., the viscous block matrix $\ViscStressMat$ and
the pressure Poisson matrices within \wBFBT, are inverted with a direct solver.
Effective preconditioners exhibit a strong clustering of eigenvalues, whereas
the convergence of Krylov methods deteriorates if the eigenvalues are spread
out.
We recognize different characteristics in the spectra associated with
\schurMass\ and \wBFBT\ preconditioning.  For \schurMass, the dominant
eigenvalues are clustered around one while
smaller eigenvalues, i.e., eigenvalues $\ll1$,
are spread out.  The behavior
for \wBFBT\ is the opposite: the dominant eigenvalues are spread out and
the smaller eigenvalues are tightly clustered around one.
Now, as the problem difficulty is increased by introducing more viscosity
anomalies in the domain, the spreading of smaller eigenvalues
associated with \schurMass\ becomes more severe (compare in
\cref{fig:robustness-eigvals} the \infigure{top row} of plots and the
\infigure{bottom row} of plots).  We postulate that this is the property that
is responsible for the deteriorating convergence with \schurMass\ that was
observed in \cref{tab:robustness-schurMass}.
With \wBFBT\ on the other hand, the spectrum remains largely unaffected by
increased sinker counts.  The clustering of smaller eigenvalues around one
remains stable, which is likely the reason for the robustness of \wBFBT.
The lower bound on the eigenvalues that we observe here numerically supports
the theoretical estimates on spectral equivalence in \cref{sec:theory}.
Therefore we find the lower bound to be sharp and, moreover, to be essential
for the robustness of the \wBFBT\ preconditioner.

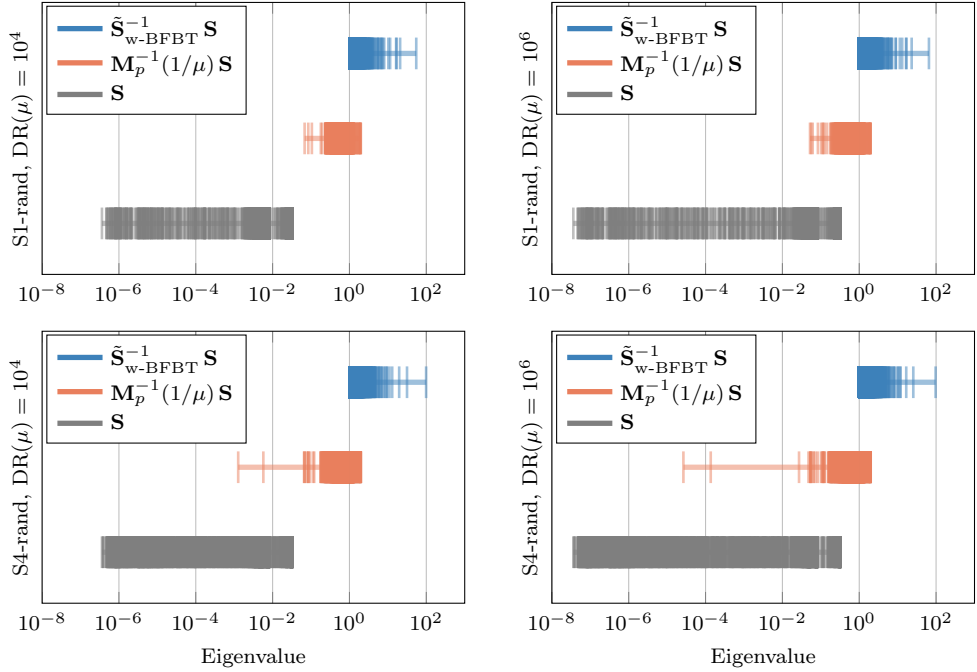
\begin{figure}\centering
  \input{inc/robustness_eigvals_1d_style}
  \caption{%
    Spectra of the Schur complement \infigure{(gray)},
    \schurMass-preconditioned Schur complement \infigure{(red)}, and
    \wBFBT-preconditioned Schur complement \infigure{(blue)};
    zero eigenvalues corresponding to the null space of the Schur complement
    matrix are omitted.
    Results for viscosities with one sinker (S1-rand) are shown in
    \infigure{top row}, and with four sinkers (S4-rand) in the
    \infigure{bottom row} of plots;
    $\dynamicratio(\visc)=10^4$ in the \infigure{left column} and
    $\dynamicratio(\visc)=10^6$ in the \infigure{right column}.
    The two-dimensional Stokes equations are discretized with
    \discrPbubblePdisc{2}{1} finite elements on a uniform triangular
    mesh consisting of 512 triangles using the FEniCS library.
    As the problem difficulty increases from one to four sinkers, the spreading
    of small eigenvalues for \schurMass\ becomes more severe, which is
    disadvantageous for solver convergence.
    For \wBFBT, the spectrum remains largely unaffected by increased sinker
    counts, which contributes to convergence that is robust with respect to
    viscosity variations.
  }
  \label{fig:robustness-eigvals}
\end{figure}

\begin{remark}
  \label{rmk:different-discretizations}
  In addition to the \discrPbubblePdisc{2}{1} discretization used for
  the results in
  \cref{fig:robustness-eigvals},
  we also calculated the spectra using \discrPP{2}{1} Taylor-Hood
  finite elements. We obtained very similar results for this
  discretization, which uses continuous
  elements to approximate the pressure.
  Therefore, both the efficacy of \wBFBT\
  as a preconditioner and the issues with \schurMass\ seem to be largely
  unaffected by the specific type of discretization, at least for the two cases
  that we tested.
\end{remark}

\begin{remark}
  \label{rmk:schurMass-visc-avg}
  The convergence of the Stokes solver with the \schurMass\
  preconditioner can be improved by approximating the heterogeneous viscosity $\visc(\bx)$
  with elementwise constants, computed by
  averaging $\visc$ over each element
  \cite{BangerthHeister15}.
  The benefit of faster convergence comes at the cost of slower asymptotic
  convergence of the discrete finite element solution and an altered
  constitutive relationship, which might be a less accurate representation of
  the physics; this is, however, problem-dependent.
  Moreover, for a nonlinear (e.g., power-law) rheology,
  elementwise averaging of
  $\visc$ can introduce non-physical, artificial disturbances in the
  effective viscosity during Newton or Picard-type nonlinear solves.
  We observed such a behavior in mantle convection simulations, which
  are governed by a nonlinear power-law rheology. Here, viscosity
  averaging led to non-physical checkerboard-like patterns upon
  convergence of the nonlinear Newton solver.
\end{remark}

\begin{remark}
  \label{rmk:wbfbt-visc-grad}
  In practice, the convergence of \wBFBT\ can be improved
  for coarse meshes, where the %
  viscosity variations over elements are large.
  This is achieved by alternative choices for the diagonal weighting matrices
  $\WbfbtScalLeftMat$ and $\WbfbtScalRightMat$ from~\eqref{eq:wbfbt} with the
  weights
  \begin{equation*}
      \bfbtweightleft(\bx)
    =
      \left(\visc^2(\bx) + \abs{\gradient\visc(\bx)}^2\right)^{\frac14}
    =
      \bfbtweightright(\bx)
  \quad\text{for all } \bx\in\Omega,
  \end{equation*}
  where
    $\abs{\cdot}$
  denotes the Euclidean norm in $\R^d$.
  These viscosity gradient-based \wBFBT\ weights have the advantage of
  performing at least as well as the pure viscosity-based weights proposed in
  \cref{sec:schur-approx}, but they exhibit superior robustness on
  coarser meshes.  They are, however, challenging to analyze
  theoretically.
\end{remark}

\section{Modifications for Dirichlet boundary conditions}
\label{sec:boundary-numerics}

In \cref{sec:comparison}, deteriorating approximation properties of
\wBFBT\ for increasing discretization order and mesh refinement level
could be observed.  The numerical experiments in
\cref{fig:comparison}, \infigure{right} did show slightly slower
convergence %
when $k$ and $\ell$ were increased.
This can stem
from \wBFBT\ representing a poor approximation to the exact Schur
complement at the boundary $\partial\Omega$ in the presence of
Dirichlet boundary conditions for the velocity.
This section investigates modifications to \wBFBT\ near a Dirichlet boundary
and aims at obtaining mesh independence and only a mild
dependence on discretization order in terms of Stokes solver convergence.

Consider the commutator that leads to the \wBFBT\ formulation in an
infinite-dimensional form:
  $\ViscStressOp\GradOp - \GradOp X \approx 0$,
where
  $\ViscStressOp$
represents the viscous stress operator,
  $\GradOp$
the gradient operator and
  $X$
the sought commuting operator.
In case of an unbounded domain $\Omega=\R^d$ and constant viscosity
$\visc\equiv1$, this commutator is exactly satisfied since
  $(\divergence\gradient)\gradient - \gradient(\divergence\gradient) = 0$.
For Dirichlet boundary conditions, the commutator does not, in general,
vanish at the boundary.  Therefore a possible source for deteriorating
Schur complement approximation properties of \wBFBT\ is a commutator mismatch
for mesh elements that are touching the boundary $\partial\Omega$.
A similar observation was also made in~\cite{ElmanTuminaro09,
ElmanSilvesterWathen14}.
A possible remedy is to modify the norm in the least-squares minimization
problem~\eqref{eq:lsc-min}, which is represented by the matrix
  $\BfbtScalLeftMat^{-1}$,
such that a damping factor is applied to the matrix entries near the boundary.
By damping the influence of the boundary in the minimization objective, more
emphasis is given to the domain interior, and the \wBFBT\ approximation is
improved.

Damping near Dirichlet boundaries can be incorporated by modifying the matrices
  $\WbfbtScalLeftMat^{-1}$
or
  $\WbfbtScalRightMat^{-1}$
of the \wBFBT\ inverse Schur complement approximation~\eqref{eq:wbfbt}.
A similar idea for \BFBT\ in a Navier--Stokes setting is presented
in~\cite{ElmanTuminaro09},
where a damping to the weighting matrix
$\BfbtScalRightMat^{-1}$ in~\eqref{eq:bfbt} is introduced to achieve mesh
independence ($\BfbtScalLeftMat^{-1}$ is not changed).  There, damping affects
the normal components of the velocity space inside mesh elements touching
$\partial\Omega$ and simply a constant damping factor of $1/10$ is set
regardless of mesh refinement $\ell$.  Also, only the discretization order
$k=2$ was considered
(in addition to \discrQQ{2}{1} and MAC discretizations).

Now, we attempt to enhance our understanding of how modifications at a Dirichlet
boundary $\partial\Omega$ influence convergence and therefore the efficacy of
\wBFBT\ as a Schur complement approximation.  Let
  $\Omega_{D} \coloneqq \bigcup_{e \in D} \Omega_e$,
  $D \coloneqq
     \{ e \mid \overline{\Omega_e}\cap\partial\Omega \neq \emptyset \}$
     be the set of all mesh elements $\Omega_e$ touching the
Dirichlet boundary.
Given values
  $\bdramp_l,\bdramp_r\ge1$,
extend the previous definition of the weights
  $\bfbtweightleft=\bfbtweightright=\sqrt{\visc}$
(see \cref{sec:schur-approx}) to a version with boundary modification:
\begin{equation}
  \label{eq:bdramp}
  \bfbtweightleft(\bx) \coloneqq
    \begin{cases}
      \bdramp_l \sqrt{\visc(\bx)} & \bx\in\Omega_{D}, \\
                \sqrt{\visc(\bx)} & \bx\notin\Omega_{D},
    \end{cases}
  \quad\text{and}\quad
  \bfbtweightright(\bx) \coloneqq
    \begin{cases}
      \bdramp_r \sqrt{\visc(\bx)} & \bx\in\Omega_{D}, \\
                \sqrt{\visc(\bx)} & \bx\notin\Omega_{D}.
    \end{cases}
\end{equation}
We obtain matrices
  $\WbfbtScalLeftMat = \VelMassLumpMat(\bfbtweightleft)$
and
  $\WbfbtScalRightMat = \VelMassLumpMat(\bfbtweightright)$
in~\eqref{eq:wbfbt} that may differ at boundary elements in $\Omega_{D}$ due
to possibly different values for $\bdramp_l$ and $\bdramp_r$.
Note that amplifying the weight functions
  $\bfbtweightleft$,
  $\bfbtweightright$
at the boundary is similar to damping at the boundary after taking the inverses
  $\WbfbtScalLeftMat^{-1}$,
  $\WbfbtScalRightMat^{-1}$.

The Stokes solver convergence under the influence of boundary amplifications
$\bdramp_l$, $\bdramp_r$ is summarized in \cref{tab:boundary}.
The table shows
that the boundary amplification is most effective when performed
non-symmetrically, i.e., either $\bdramp_l>1$ or $\bdramp_r>1$ but not both.
Further, we deduce that with higher mesh refinement level $\ell$, the boundary
amplification should increase roughly proportional to $2^\ell$
(or proportional to the reciprocal element size, here $h^{-1}=2^\ell$).
Similar observations can be made for the discretization order $k$,
i.e., amplification needs to increase for larger $k$ to avoid higher iteration
counts.
These implications were made based on extensive numerical
experiments for which \cref{tab:boundary} serves as a representative summary.

\begin{table}
  \caption[Influence of boundary modification factors $\bdramp_l$,
           $\bdramp_r$ on the Stokes solver convergence with \wBFBT.]{%
    Influence of boundary modification factors $\bdramp_l$, $\bdramp_r$
    on the Stokes solver convergence with \wBFBT\ for discretizations:
    $k=2$, $\ell=5,\ldots,7$
    (see \subref{tab:boundary-2-5},
         \subref{tab:boundary-2-6},
         \subref{tab:boundary-2-7})
    and
    $k=2,\ldots,5$, $\ell=5$
    (see \subref{tab:boundary-2-5},
         \subref{tab:boundary-3-5},
         \subref{tab:boundary-4-5},
         \subref{tab:boundary-5-5}).
    Reported are the number of GMRES iterations for $10^{-6}$ residual
    reduction for the problem S16-rand, $\dynamicratio(\visc)=10^6$.
    \infigure{Colors} highlight solves within
    $\sim$5\% of iterations above
    the lowest iteration count.
    Increase of mesh refinement level $\ell$ or discretization order $k$
    demands larger boundary amplification $\bdramp_r$ or $\bdramp_l$ to
    maintain fast convergence.
  }
  \label{tab:boundary}
  \def\ccA{\cellcolor{jr@red!40}}
  \def\ccB{\cellcolor{jr@purple!40}}
  \def\ccC{\cellcolor{jr@blue!40}}
  \def\ccD{\cellcolor{jr@green!40}}
  \def\ccE{\cellcolor{jr@purple!40}}
  \def\ccF{\cellcolor{jr@pink!40}}
  \def\ccG{\cellcolor{jr@orange!40}}
  \def\ccH{\cellcolor{jr@brown!40}}
  \begin{subtable}[b]{0.49\columnwidth}
    \caption{$k=2$, $\ell=5$}
    \label{tab:boundary-2-5}
    \centering
    \footnotesize
    \setlength{\tabcolsep}{0.5em}  %
    \begin{tabular}{lcccccc}
      \toprule
      $\bdramp_l \setminus \bdramp_r$ & $1$ & $2$ & $4$ & $8$ & $16$ & $32$ \\
      \midrule
      \phantom{1}1 &\ccB 33 &\ccB 33 &\ccB 34 &\ccB 34 &\ccB 34 &\ccB 35 \\
      \phantom{1}2 &\ccB 33 &\ccB 33 &\ccB 34 &\ccB 34 &\ccB 34 &\ccB 34 \\
      \phantom{1}4 &\ccB 33 &\ccB 34 &\ccB 34 &     36 &     38 &     39 \\
      \phantom{1}8 &\ccB 34 &\ccB 34 &     36 &     39 &     43 &     44 \\
                16 &\ccB 34 &\ccB 34 &     38 &     43 &     46 &     49 \\
                32 &\ccB 34 &\ccB 34 &     39 &     44 &     49 &     53 \\
      \bottomrule
    \end{tabular}
  \end{subtable}
  \hfill
  \begin{subtable}[b]{0.49\columnwidth}
    \caption{$k=3$, $\ell=5$}
    \label{tab:boundary-3-5}
    \centering
    \footnotesize
    \setlength{\tabcolsep}{0.5em}  %
    \begin{tabular}{lcccccc}
      \toprule
      $\bdramp_l \setminus \bdramp_r$ & $1$ & $2$ & $4$ & $8$ & $16$ & $32$ \\
      \midrule
      \phantom{1}1 &     41 &\ccF 38 &\ccF 37 &\ccF 37 &\ccF 37 &\ccF 37 \\
      \phantom{1}2 &\ccF 38 &\ccF 37 &\ccF 38 &\ccF 38 &     39 &     39 \\
      \phantom{1}4 &\ccF 37 &\ccF 38 &     40 &     42 &     44 &     46 \\
      \phantom{1}8 &\ccF 36 &\ccF 38 &     42 &     47 &     50 &     51 \\
                16 &\ccF 37 &     39 &     44 &     50 &     53 &     56 \\
                32 &\ccF 37 &     39 &     45 &     51 &     56 &     59 \\
      \bottomrule
    \end{tabular}
  \end{subtable}
  \\
  \begin{subtable}[b]{0.49\columnwidth}
    \caption{$k=2$, $\ell=6$}
    \label{tab:boundary-2-6}
    \centering
    \footnotesize
    \setlength{\tabcolsep}{0.5em}  %
    \begin{tabular}{lcccccc}
      \toprule
      $\bdramp_l \setminus \bdramp_r$ & $1$ & $2$ & $4$ & $8$ & $16$ & $32$ \\
      \midrule
      \phantom{1}1 &     37 &\ccC 34 &\ccC 33 &\ccC 34 &\ccC 34 &\ccC 34 \\
      \phantom{1}2 &\ccC 34 &\ccC 34 &\ccC 34 &\ccC 34 &\ccC 34 &\ccC 34 \\
      \phantom{1}4 &\ccC 33 &\ccC 33 &\ccC 34 &\ccC 35 &     36 &     37 \\
      \phantom{1}8 &\ccC 34 &\ccC 34 &\ccC 35 &     38 &     39 &     39 \\
                16 &\ccC 34 &\ccC 34 &     36 &     39 &     40 &     41 \\
                32 &\ccC 34 &\ccC 34 &     37 &     39 &     41 &     42 \\
      \bottomrule
    \end{tabular}
  \end{subtable}
  \hfill
  \begin{subtable}[b]{0.49\columnwidth}
    \caption{$k=4$, $\ell=5$}
    \label{tab:boundary-4-5}
    \centering
    \footnotesize
    \setlength{\tabcolsep}{0.5em}  %
    \begin{tabular}{lcccccc}
      \toprule
      $\bdramp_l \setminus \bdramp_r$ & $1$ & $2$ & $4$ & $8$ & $16$ & $32$ \\
      \midrule
      \phantom{1}1 &     44 &     39 &\ccG 36 &\ccG 36 &\ccG 36 &\ccG 36 \\
      \phantom{1}2 &     39 &     39 &     39 &     40 &     41 &     41 \\
      \phantom{1}4 &\ccG 36 &     39 &     43 &     47 &     49 &     51 \\
      \phantom{1}8 &\ccG 36 &     40 &     47 &     52 &     56 &     58 \\
                16 &\ccG 36 &     41 &     49 &     56 &     60 &     63 \\
                32 &\ccG 36 &     41 &     50 &     58 &     63 &     66 \\
      \bottomrule
    \end{tabular}
  \end{subtable}
  \\
  \begin{subtable}[b]{0.49\columnwidth}
    \caption{$k=2$, $\ell=7$}
    \label{tab:boundary-2-7}
    \centering
    \footnotesize
    \setlength{\tabcolsep}{0.5em}  %
    \begin{tabular}{lcccccc}
      \toprule
      $\bdramp_l \setminus \bdramp_r$ & $1$ & $2$ & $4$ & $8$ & $16$ & $32$ \\
      \midrule
      \phantom{1}1 &     45 &     37 &\ccD 34 &\ccD 34 &\ccD 34 &\ccD 34 \\
      \phantom{1}2 &     37 &\ccD 36 &\ccD 35 &\ccD 36 &\ccD 36 &\ccD 36 \\
      \phantom{1}4 &\ccD 34 &\ccD 36 &     38 &     39 &     40 &     41 \\
      \phantom{1}8 &\ccD 34 &\ccD 36 &     39 &     42 &     44 &     44 \\
                16 &\ccD 34 &\ccD 36 &     40 &     44 &     45 &     46 \\
                32 &\ccD 34 &\ccD 36 &     41 &     44 &     46 &     47 \\
      \bottomrule
    \end{tabular}
  \end{subtable}
  \hfill
  \begin{subtable}[b]{0.49\columnwidth}
    \caption{$k=5$, $\ell=5$}
    \label{tab:boundary-5-5}
    \centering
    \footnotesize
    \setlength{\tabcolsep}{0.5em}  %
    \begin{tabular}{lcccccc}
      \toprule
      $\bdramp_l \setminus \bdramp_r$ & $1$ & $2$ & $4$ & $8$ & $16$ & $32$ \\
      \midrule
      \phantom{1}1 &     63 &     53 &     46 &\ccH 43 &\ccH 43 &\ccH 44 \\
      \phantom{1}2 &     53 &     51 &     51 &     51 &     52 &     53 \\
      \phantom{1}4 &     47 &     51 &     55 &     59 &     62 &     64 \\
      \phantom{1}8 &\ccH 44 &     51 &     59 &     65 &     69 &     72 \\
                16 &\ccH 43 &     52 &     62 &     69 &     75 &     78 \\
                32 &\ccH 44 &     53 &     64 &     72 &     78 &     82 \\
      \bottomrule
    \end{tabular}
  \end{subtable}
\end{table}

\begin{remark}
  \label{rmk:damping-vs-spectral-equiv}
  The theoretical derivations of spectral equivalence from
  \cref{sec:theory} and the necessity for damping at Dirichlet
  boundaries appear inconsistent.
  However, spectral equivalence was shown in
  infinite dimensions whereas boundary damping is applied to the
  discretized problem.  Therefore, we believe that the necessity for
  damping is introduced through the discretization.  It still remains
  an open question what might be causing the slowdown in convergence
  that is avoided by damping.
\end{remark}

\section{Parallel hybrid spectral-geometric-algebraic multigrid (HMG) for
\wBFBT}
\label{sec:hmg}

Two aspects of the Stokes preconditioner with \wBFBT\ have not been discussed
yet.  One is the approximation of the inverse viscous block
  $\ViscStressPcMat^{-1}$
required in~\eqref{eq:precond-stokes} and the other is the approximation of
inverses
  $\PressPoissonLeftPcMat^{-1} \approx
    \PressPoissonLeftMat^{-1} \coloneqq
    (\DivMat\WbfbtScalLeftMat^{-1}\GradMat)^{-1}$
and
  $\PressPoissonRightPcMat^{-1} \approx
    \PressPoissonRightMat^{-1} \coloneqq
    (\DivMat\WbfbtScalRightMat^{-1}\GradMat)^{-1}$
in \eqref{eq:wbfbt}.
These approximations are crucial for overall Stokes solver performance and
scalability and are addressed in this section.
For brevity, we limit our discussion to
$\PressPoissonRightPcMat^{-1}$; the results also hold for
$\PressPoissonLeftPcMat^{-1}$.

The approximation of the inverse viscous block $\ViscStressPcMat^{-1}$
is well suited for multigrid V-cycles.  To this end,
in~\cite{RudiMalossiIsaacEtAl15}, we developed a hybrid
spectral-geometric-algebraic multigrid (HMG) method, which exhibits
extreme parallel scalability and retains nearly optimal algorithmic
scalability (see \cref{sec:scalability} for scalability results).
While traversing the HMG hierarchy shown in \cref{fig:hmg}, HMG initially
reduces the discretization order (spectral multigrid); after arriving at order
one, it continues by coarsening mesh elements (geometric multigrid); once the
degrees of freedom fall below a threshold, algebraic multigrid (AMG) carries
out further coarsening until a direct solve can be computed efficiently
(\cite{SundarBirosBursteddeEtAl12}).
Parallel forest-of-octrees algorithms, implemented in the p4est parallel
adaptive mesh refinement library \cite{BursteddeGhattasGurnisEtAl10,
BursteddeWilcoxGhattas11}, are used for efficient, scalable mesh
refinement/coarsening, mesh balancing, and repartitioning in the geometric HMG
phase.
During parallel geometric coarsening, the number of compute cores and the size
of the MPI communicator is reduced successively to minimize communication.
Re-discretization of the differential equations is performed on each coarser
spectral and geometric level. The viscosity values in each element are stored
at the quadrature points of the velocity discretization, and are thus
local to each element. The viscosity coarsening is done level-by-level
during the setup phase. The coarsening operator is the
adjoint of the refinement operator, which performs element-wise
interpolation. This adjoint is computed with respect to the
$\LL$-inner products, and since viscosity values are not shared
amongst elements, this does not require (an approximation of) a global
mass matrix solve.
The transition from geometric to algebraic multigrid is done at a
sufficiently small core count and small MPI
communicator.  AMG continues to further reduce problem size (via Galerkin
coarse grid projection) and the number of cores down to a single core for the
direct solver.

The operator $\PressPoissonRightMat$ can be regarded as a discrete,
variable-coefficient Poisson operator on the discontinuous pressure space
\discrPdisc{k-1} with Neumann boundary conditions.
Therefore, multigrid V-cycles can also be employed to approximate
the inverse
$\PressPoissonRightPcMat^{-1}$.
However, it turned out to be problematic to apply multigrid coarsening
directly due to the discontinuous, modal discretization of the pressure.
We took a novel approach in~\cite{RudiMalossiIsaacEtAl15} by
considering the underlying infinite-dimensional, variable-coefficient Poisson
operator, where the coefficient is derived from the diagonal weighting matrix
(here, $\WbfbtScalLeftMat^{-1}$ or $\WbfbtScalRightMat^{-1}$).
Then we re-discretize with continuous, nodal high-order finite elements in
\discrQ{k}. An alternative would be to use \discrQ{k-1}, but we prefer
to use \discrQ{k} since the corresponding data
structures are readily available from the discretization of the velocity.
Hence, this choice %
avoids the setup cost related to discretization-specific parameters and their
storage.  Additionally, the HMG hierarchy of the preconditioner acting on the
velocity can be partially reused, again saving setup time and memory.
This continuous, nodal discretization of the Poisson operator is
then approximately inverted with an HMG V-cycle that is similar to the one
described above for the inverse viscous block approximation
$\ViscStressPcMat^{-1}$.  Additional smoothing is applied in the discontinuous
pressure space (\cref{fig:hmg}, \infigure{green level}) to account for high
frequency modes in residuals that are introduced through projections
between \discrQ{k} and \discrPdisc{k-1}.

Our hybrid multigrid method combines high-order
$\LL$-restrictions/interpolations and employs Chebyshev-accelerated
point-Jacobi smoothers.  This results in optimal or nearly optimal
algorithmic multigrid performance (see \cref{sec:scalability}), i.e.,
iteration numbers are independent of mesh size and only mildly
dependent on discretization order, while maintaining robustness with
respect to highly heterogeneous coefficients.
In addition, the efficacy of the HMG preconditioner does not
deteriorate with increasing core counts, because the spectral and geometric
multigrid is by construction independent of the number of cores and AMG is
invoked for prescribed small problem sizes on essentially fixed small core
counts.

For all numerical experiments presented here, three pre- and
post-smoothing iterations with a Chebyshev accelerated point-Jacobi
smoother
are performed.  AMG is always invoked on just one MPI
rank after geometrically coarsening the uniform mesh to $\ell=2$ with 64
elements (using a direct solver for such small problems is also reasonable).
GMRES is restarted after every 100 iterations throughout all experiments.
PETSc's \cite{BalayAbhyankarAdamsEtAl15} implementations of Chebyshev
acceleration, direct solver, AMG (called GAMG), and GMRES are used.

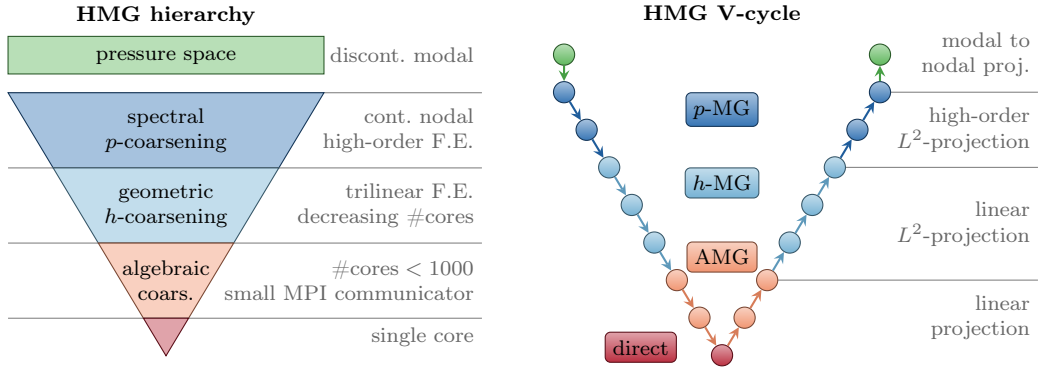
\begin{figure}\centering
  \def\colorBBT{jr@green}
  \def\colorSMG{jr@div-darkblue}
  \def\colorGMG{jr@div-medblue}
  \def\colorAMG{jr@div-medred}
  \def\colorDirect{jr@div-darkred}
  \def\colorSaturation{40}
  \def\colorSaturationMax{90}
  \def\W{0.6}
  \def\Wcomment{4.2}
  \def\Lbbt{8}
  \def\LbbtOffset{0.5}
  \def\Ldirect{1}
  \begin{tikzpicture}[scale=0.50]
    \footnotesize
    \tikzset{
        mglevel/.style={
            draw=#1!50!black,
            fill=#1!\colorSaturation},
        mglabel/.style={
            draw=#1!50!black,thin,rectangle,rounded corners=2pt,
            top color=#1!\colorSaturation,bottom color=#1!\colorSaturationMax,
            text=black,minimum width=0.9cm,minimum height=0.4cm},
    }
    \def\Lsmg{7}
    \def\Lgmg{5}
    \def\Lamg{3}
    \coordinate (xmin) at ($-\W*(\Lsmg,0)$);
    \coordinate (xmax) at ($+\W*(\Lsmg,0) + (\Wcomment,0)$);
    \coordinate (ybbt) at ($0.5*(0,\Lbbt+\Lsmg+2*\LbbtOffset)$);
    \node[above,black] at (0,\Lbbt+\LbbtOffset+0.1) {\textbf{HMG hierarchy}};
    \filldraw[mglevel=\colorBBT]
      (-\W*\Lsmg,\Lbbt+\LbbtOffset) -- (+\W*\Lsmg,\Lbbt+\LbbtOffset) --
      (+\W*\Lsmg,\Lsmg+\LbbtOffset) -- (-\W*\Lsmg,\Lsmg+\LbbtOffset) --
      cycle;
    \draw[gray]
      ($(xmin) + (0,\Lsmg)$) --
      ($(xmax) + (0,\Lsmg)$);
    \filldraw[mglevel=\colorSMG]
      (-\W*\Lsmg,\Lsmg) -- (+\W*\Lsmg,\Lsmg) --
      (+\W*\Lgmg,\Lgmg) -- (-\W*\Lgmg,\Lgmg) -- cycle;
    \draw[gray]
      ($(xmin) + (0,\Lgmg)$) -- ($(xmax) + (0,\Lgmg)$);
    \filldraw[mglevel=\colorGMG]
      (-\W*\Lgmg,\Lgmg) -- (+\W*\Lgmg,\Lgmg) --
      (+\W*\Lamg,\Lamg) -- (-\W*\Lamg,\Lamg) -- cycle;
    \draw[gray]
      ($(xmin) + (0,\Lamg)$) -- ($(xmax) + (0,\Lamg)$);
    \filldraw[mglevel=\colorAMG]
      (-\W*\Lamg,\Lamg) -- (+\W*\Lamg,\Lamg) --
      (+\W*\Ldirect,\Ldirect) -- (-\W*\Ldirect,\Ldirect) -- cycle;
    \draw[gray]
      ($(xmin) + (0,\Ldirect)$) -- ($(xmax) + (0,\Ldirect)$);
    \filldraw[mglevel=\colorDirect]
      (-\W*\Ldirect,\Ldirect) -- (+\W*\Ldirect,\Ldirect) -- (0,0) -- cycle;
    \node[black,align=center] at ($(ybbt) - (0,0.05)$)
      {pressure space};
    \node[black,align=center] at ($0.5*(0,\Lsmg+\Lgmg)$)
      {spectral \\ $p$-coarsening};
    \node[black,align=center] at ($0.5*(0,\Lgmg+\Lamg)$)
      {geometric \\ $h$-coarsening};
    \node[black,align=center] at ($0.5*(0,\Lamg+\Ldirect)$)
      {algebraic \\ coars.};
    \node[jr@medgray,anchor=east,align=right]
      at ($(xmax) + (ybbt)$)
      {discont.\ modal};
    \node[jr@medgray,anchor=east,align=right]
      at ($(xmax) + 0.5*(0,\Lsmg+\Lgmg)$)
      {cont.\ nodal \\ high-order F.E.};
    \node[jr@medgray,anchor=east,align=right]
      at ($(xmax) + 0.5*(0,\Lgmg+\Lamg)$)
      {trilinear F.E. \\ decreasing \#cores};
    \node[jr@medgray,anchor=east,align=right]
      at ($(xmax) + 0.5*(0,\Lamg+\Ldirect)$)
      {$\text{\#cores} < 1000$ \\ small MPI communicator};
    \node[jr@medgray,anchor=east,align=right]
      at ($(xmax) + 0.5*(0,\Ldirect)$)
      {single core};
    \node [circle,minimum width=0.2cm] at (0,0) {};
  \end{tikzpicture}
  \hspace{2em} 
  \begin{tikzpicture}[scale=0.50]
    \footnotesize
    \tikzset{
        mglevel/.style={
            draw=#1!50!black,
            fill=#1!\colorSaturation},
        mglabel/.style={
            draw=#1!50!black,thin,rectangle,rounded corners=2pt,
            top color=#1!\colorSaturation,bottom color=#1!\colorSaturationMax,
            text=black,minimum width=0.9cm,minimum height=0.4cm},
        mgsmooth/.style={
            draw=#1!50!black,thin,circle,
            top color=#1!\colorSaturation,bottom color=#1!\colorSaturationMax,
            text=black,minimum width=0.2cm},
        mgrestrict/.style={
            #1!90!black,thick,->,>=stealth},
        mginterp/.style={
            #1!90!black,thick,<-,>=stealth},
    }
    \def\Lsmg{7}
    \def\Lgmg{5}
    \def\Lamg{2}
    \coordinate (xmin) at ($-\W*(\Lsmg,0)$);
    \coordinate (xmax) at ($+\W*(\Lsmg,0) + (\Wcomment,0)$);
    \coordinate (ybbt) at ($0.5*(0,\Lbbt+\Lsmg+2*\LbbtOffset)$);
    \node[above,black] at (0,\Lbbt+\LbbtOffset+0.1) {\textbf{HMG V-cycle}};
    \draw[gray] (+\W*\Lsmg,\Lsmg) -- ($(xmax) + (0,\Lsmg)$);
    \draw[gray] (+\W*\Lgmg,\Lgmg) -- ($(xmax) + (0,\Lgmg)$);
    \draw[gray] (+\W*\Lamg,\Lamg) -- ($(xmax) + (0,\Lamg)$);
    \node (bD1) [mgsmooth=\colorBBT] at ($(-\W*\Lsmg,0) + (ybbt)$) {};
    \node (bU1) [mgsmooth=\colorBBT] at ($(+\W*\Lsmg,0) + (ybbt)$) {};
    \node (pD1) [mgsmooth=\colorSMG] at (-\W*\Lsmg,\Lsmg) {};
    \node (pU1) [mgsmooth=\colorSMG] at (+\W*\Lsmg,\Lsmg) {};
    \node (pD2) [mgsmooth=\colorSMG] at (-\W*\Lsmg+\W,\Lsmg-1) {};
    \node (pU2) [mgsmooth=\colorSMG] at (+\W*\Lsmg-\W,\Lsmg-1) {};
    \node (hD1) [mgsmooth=\colorGMG] at (-\W*\Lgmg,\Lgmg) {};
    \node (hU1) [mgsmooth=\colorGMG] at (+\W*\Lgmg,\Lgmg) {};
    \node (hD2) [mgsmooth=\colorGMG] at (-\W*\Lgmg+\W,\Lgmg-1) {};
    \node (hU2) [mgsmooth=\colorGMG] at (+\W*\Lgmg-\W,\Lgmg-1) {};
    \node (hD3) [mgsmooth=\colorGMG] at (-\W*\Lgmg+2*\W,\Lgmg-2) {};
    \node (hU3) [mgsmooth=\colorGMG] at (+\W*\Lgmg-2*\W,\Lgmg-2) {};
    \node (aD1) [mgsmooth=\colorAMG] at (-\W*\Lamg,\Lamg) {};
    \node (aU1) [mgsmooth=\colorAMG] at (+\W*\Lamg,\Lamg) {};
    \node (aD2) [mgsmooth=\colorAMG] at (-\W*\Lamg+\W,\Lamg-1) {};
    \node (aU2) [mgsmooth=\colorAMG] at (+\W*\Lamg-\W,\Lamg-1) {};
    \node (dir) [mgsmooth=\colorDirect] at (0,0) {};
    \draw [mgrestrict=\colorBBT] (bD1) -- (pD1);
    \draw [mgrestrict=\colorSMG] (pD1) -- (pD2);
    \draw [mgrestrict=\colorSMG] (pD2) -- (hD1);
    \draw [mgrestrict=\colorGMG] (hD1) -- (hD2);
    \draw [mgrestrict=\colorGMG] (hD2) -- (hD3);
    \draw [mgrestrict=\colorGMG] (hD3) -- (aD1);
    \draw [mgrestrict=\colorAMG] (aD1) -- (aD2);
    \draw [mgrestrict=\colorAMG] (aD2) -- (dir);
    \draw [mginterp=\colorBBT] (bU1) -- (pU1);
    \draw [mginterp=\colorSMG] (pU1) -- (pU2);
    \draw [mginterp=\colorSMG] (pU2) -- (hU1);
    \draw [mginterp=\colorGMG] (hU1) -- (hU2);
    \draw [mginterp=\colorGMG] (hU2) -- (hU3);
    \draw [mginterp=\colorGMG] (hU3) -- (aU1);
    \draw [mginterp=\colorAMG] (aU1) -- (aU2);
    \draw [mginterp=\colorAMG] (aU2) -- (dir);
    \node[mglabel=\colorSMG,anchor=north] at (0,\Lsmg) {$p$-MG};
    \node[mglabel=\colorGMG,anchor=north] at (0,\Lgmg) {$h$-MG};
    \node[mglabel=\colorAMG,anchor=south] at (0,\Lamg+0.2) {AMG};
    \node[mglabel=\colorDirect,anchor=south east]
      at ($(-\W*\Ldirect-\W,-0.2)$) {direct};
    \node[jr@medgray,anchor=east,align=right]
      at ($(xmax) + (ybbt)$)
      {modal to \\ nodal proj.};
    \node[jr@medgray,anchor=east,align=right]
      at ($(xmax) + 0.5*(0,\Lsmg+\Lgmg)$)
      {high-order \\ $\LL$-projection};
    \node[jr@medgray,anchor=east,align=right]
      at ($(xmax) + 0.5*(0,\Lgmg+\Lamg)$)
      {linear \\ $\LL$-projection};
    \node[jr@medgray,anchor=east,align=right]
      at ($(xmax) + (0,\Ldirect)$)
      {linear \\ projection};
  \end{tikzpicture}
  \caption{%
    Hybrid spectral-geometric-algebraic multigrid (HMG).
    \infigure{Left}: Illustration of multigrid hierarchy.
    From top to bottom, first, the multigrid levels are obtained by spectral
    coarsening \infigure{(dark blue)}.
    Next, the mesh is geometrically coarsened and repartitioned on successively
    fewer cores to minimize communication \infigure{(light blue)}.  Finally,
    AMG further reduces problem size and core count \infigure{(light red)}.
    The multigrid hierarchy for the pressure Poisson
    operator~$\PressPoissonRightMat$ additionally involves smoothing in the
    discontinuous, modal pressure space \infigure{(green)}. The
    projection from the discontinuous, modal to a continuous finite element
    nodal basis uses a lumped mass matrix in the nodal space to avoid the
    global mass matrix system solve.
    \infigure{Right}: The multigrid V-cycle consists of smoothing at each level
    of the hierarchy \infigure{(circles)} and intergrid transfer operators
    (\infigure{arrows downward} for restriction and \infigure{arrows upward}
    for interpolation).  To enhance efficacy of the the V-cycle as a
    preconditioner, different types of projection operators are employed for
    these operators depending on the phase within the V-cycle.
  }
  \label{fig:hmg}
\end{figure}

\section{Algorithmic and parallel scalability for \HMGwBFBT\ Stokes
preconditioner}
\label{sec:scalability}

After establishing the robustness of the Stokes solver with
\wBFBT\ preconditioning in theory (\cref{sec:theory}) and numerically
(\cref{sec:robustness}), addressing issues associated with Dirichlet
boundary conditions (\cref{sec:boundary-numerics}), and introducing an
effective and scalable multigrid method (\cref{sec:hmg}), in this
section we finally study the scalability of the Stokes solver building on
\HMGwBFBT.  One aspect of scalability is algorithmic scalability,
i.e., the dependence of Krylov iterations on the mesh resolution
and the discretization order.  The second aspect is parallel
scalability of the implementation, i.e., runtime measured on
increasing numbers of compute cores.  Studying both aspects is
required to fully assess the performance of a solver at scale.

The algorithmic scalability in \cref{tab:scalability-alg} shows results for the
Stokes solver as well as its individual components by reporting iteration numbers
for solving the systems
  $\ViscStressMat\velvec=\rhsvelvec$
and
  $\PressPoissonRightMat\pressvec=\rhspressvec$.
Studying the individual components allows us to observe HMG performance in
isolation and to compare it to the algorithmic scalability of the full Stokes
system, which is indicative of the quality of the \wBFBT\ Schur complement
approximation.
All systems
are solved
with preconditioned GMRES down to a relative tolerance of $10^{-6}$.  The
preconditioners for $\ViscStressMat$ and $\PressPoissonRightMat$ are
HMG-V-cycles as described in \cref{sec:hmg}.  For the \wBFBT\ preconditioner,
we set a constant left boundary amplification $\bdramp_l=1$ and vary the right
boundary amplification $\bdramp_r$ according to results from
\cref{sec:boundary-numerics}.
The iteration counts in \cref{tab:scalability-alg-level} show textbook
mesh independence when increasing the level of refinement $\ell$.
This holds for each component, $\ViscStressMat$ and
$\PressPoissonRightMat$, and also the whole Stokes solver, and hence
we conclude that the Schur complement approximation by \wBFBT\ is
mesh-independent.
When the discretization order $k$ is increased, the iteration counts presented in
\cref{tab:scalability-alg-order} increase mildly.  The convergence of both
components $\ViscStressMat$ and $\PressPoissonRightMat$ exhibits a moderate
dependence on $k$.  Since the increase in number of iterations is sightly
larger for the full Stokes solve than for $\ViscStressMat$ and
$\PressPoissonRightMat$, we suspect a mild deterioration of \wBFBT\ as a Schur
complement approximation.

\begin{table}
  \caption[Algorithmic scalability for Stokes solver]{%
    Algorithmic scalability for Stokes solver with \HMGwBFBT\ preconditioning
    while
    \subref{tab:scalability-alg-level} varying mesh refinement level $\ell$ and
    \subref{tab:scalability-alg-order} varying discretization order $k$
    (problem S16-rand, $\dynamicratio(\visc)=10^6$).
    Computational cost is expressed in number of GMRES iterations
    (abbreviated by It.) for full Stokes solve ($10^{-6}$ residual reduction).
    Left boundary amplification for $\WbfbtScalLeftMat$ is fixed to
    $\bdramp_l=1$; right boundary amplification $\bdramp_r$ for
    $\WbfbtScalRightMat$ varies.
    Additionally, the numbers of GMRES iterations for solving only the
    sub-systems $\ViscStressMat\velvec=\rhsvelvec$ and
    $\PressPoissonRightMat\pressvec=\rhspressvec$ are given for demonstration
    of HMG efficacy
    (here, $\rhsvelvec$ is the right-hand side of the momentum equation and
    $\rhspressvec$ is the discreet representation of $\divergence\rhsvel$;
    however, random right-hand sides would give similar convergence results).
  }
  \label{tab:scalability-alg}
  \centering
  \begin{subtable}[b]{0.49\columnwidth} 
    \centering
    \caption{Algorithmic scalability (fixed order $k=2$)}
    \label{tab:scalability-alg-level}
    \footnotesize
    \setlength{\tabcolsep}{0.5em}  
    \begin{tabular}{ccrlrlrl}
      \toprule
      $\ell$ & $a_r$ &
        $\vel$-DOF   & It. &
        $\press$-DOF & It. &
        DOF          & It. \\
      & &
        $[\times 10^6]$ & $\ViscStressMat$\hspace{-1em} &
        $[\times 10^6]$ & $\PressPoissonRightMat$\hspace{-1em} &
        $[\times 10^6]$ & \textbf{Stokes} \\
      \midrule
       4 &  1 &     0.11 & 18 &    0.02 & 8 &     0.12 & 40 \\
       5 &  2 &     0.82 & 18 &    0.13 & 7 &     0.95 & 33 \\
       6 &  4 &     6.44 & 18 &    1.05 & 6 &     7.49 & 33 \\
       7 &  8 &    50.92 & 18 &    8.39 & 6 &    59.31 & 34 \\
       8 & 16 &   405.02 & 18 &   67.11 & 6 &   472.12 & 34 \\
       9 & 32 &  3230.67 & 18 &  536.87 & 6 &  3767.54 & 34 \\
      10 & 64 & 25807.57 & 18 & 4294.97 & 6 & 30102.53 & 34 \\
      \bottomrule
    \end{tabular}
  \end{subtable}
  \hfill 
  \begin{subtable}[b]{0.49\columnwidth} 
    \centering
    \caption{Algorithmic scalability (fixed level $\ell=5$)}
    \label{tab:scalability-alg-order}
    \footnotesize
    \setlength{\tabcolsep}{0.5em}  
    \begin{tabular}{ccrlrlrl}
      \toprule
      $k$ & $a_r$ &
        $\vel$-DOF   & It. &
        $\press$-DOF & It. &
        DOF          & It. \\
      & &
        $[\times 10^6]$ & $\ViscStressMat$ &
        $[\times 10^6]$ & $\PressPoissonRightMat$ &
        $[\times 10^6]$ & \textbf{Stokes} \\
      \midrule
      2 &   2 &  0.82 & 18 & 0.13 & \phantom{1}7 &  0.95 & 33 \\
      3 &   4 &  2.74 & 20 & 0.32 & \phantom{1}8 &  3.07 & 37 \\
      4 &   8 &  6.44 & 20 & 0.66 & \phantom{1}7 &  7.10 & 36 \\
      5 &  16 & 12.52 & 23 & 1.15 &           12 & 13.67 & 43 \\
      6 &  32 & 21.56 & 23 & 1.84 &           12 & 23.40 & 50 \\
      7 &  64 & 34.17 & 22 & 2.75 &           10 & 36.92 & 54 \\
      8 & 128 & 50.92 & 22 & 3.93 &           10 & 54.86 & 67 \\
      \bottomrule
    \end{tabular}
  \end{subtable}
\end{table}

The following parallel scalability results complement the results presented
in~\cite{RudiMalossiIsaacEtAl15}, where scalability to millions of
threads was demonstrated on
IBM's BlueGene/Q architecture, which
is specifically designed for large-scale supercomputers and thus
differs from conventional clusters.
The parallel scalability results here were obtained on the full Lonestar~5
peta-scale system, which represents a rather conventional cluster,
housed at the Texas Advanced Computing Center (TACC).
The Lonestar~5 supercomputer entered production in January 2016 and
is a Cray XC40 system consisting of 1252 compute nodes.
Each node is equipped with two Intel Haswell 12-core processors
(Xeon E5-2680v3) and 64~GBytes of memory. Inter-node communication is
based on an Aries Dragonfly topology network that provides dynamic
routing and thus enables optimal use of the system bandwidth.

Results for weak scalability (DOF/core fixed to $\sim$1 million)
in \cref{fig:scalability-weak} show that the Stokes solver with
\wBFBT\ \infigure{(blue curve)} maintains 90\% parallel efficiency
over a 618-fold increase in degrees of freedom along with cores.  Even
for the setup of the Stokes solver \infigure{(green curve)}, which
mainly involves generation of the HMG hierarchy, we observe 71\%
parallel efficiency.  These are
excellent results for such a
complex implicit multilevel solver with optimal algorithmic
performance (when the mesh is refined, or nearly algorithmically
optimal when the order is increased) and with convergence that is
independent of the number of cores.

Finally, \cref{fig:scalability-strong} reports strong scalability
results (overall DOF fixed to 59 million) and how the number of OpenMP
threads substituting MPI ranks influences speedup.%
\footnote{%
  Even though the processors of Lonestar~5 support two threads per physical
  core (Intel Hyper-Threading Technology), assigning more than one
  OpenMP thread per core did not improve performance.
}
Over the
78-fold increase from 48 to 3744 cores, efficiency reduces moderately,
to a worst-case 68\% for 24$\times$OMP1.  However, note that in the
largest run with 29,640 cores, the granularity is only
$\sim$2000~DOF/core, which is extremely challenging for strong
scalability.  In this case, due to the increased communication volume,
overlapping with decreased amounts of computation becomes impossible
and communication dominates the runtime.  This behavior is expected
for an implicit solver, especially for a multilevel
method that does not sacrifice algorithmic optimality for parallel
scalability.

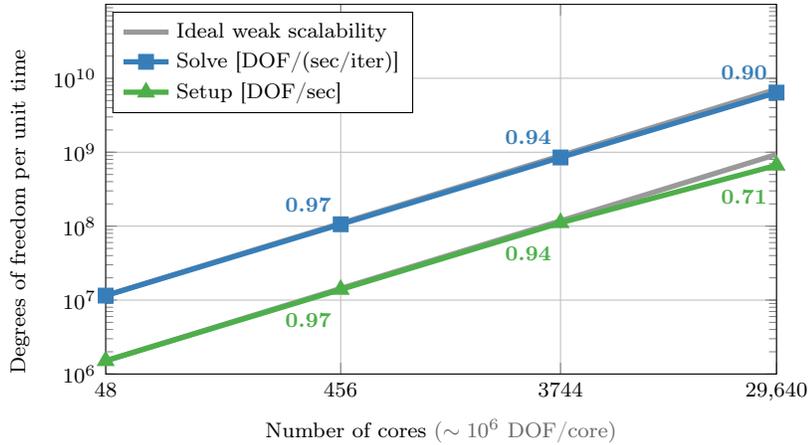
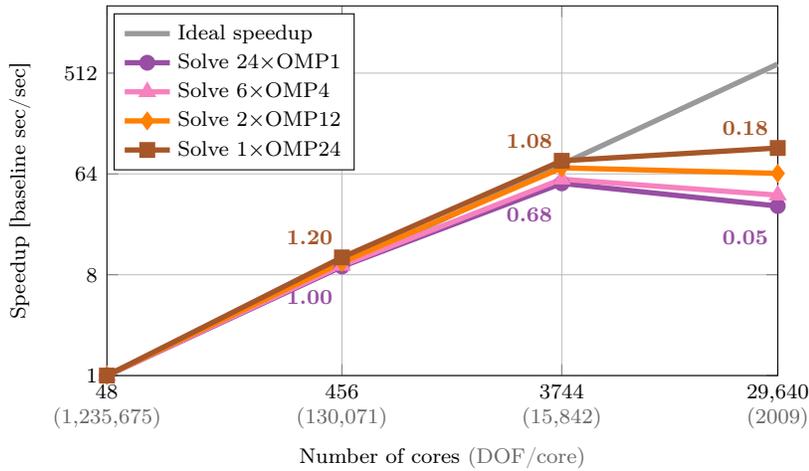
\begin{figure}\centering
  \tikzset{
    scalgraph/.style={
      color=#1, line width=2pt},
    scallabeltop/.style={
      text=#1, font=\footnotesize\bf,
      anchor=south east, yshift=+0.05cm},
    scallabelbottom/.style={
      text=#1, font=\footnotesize\bf,
      anchor=north east, yshift=-0.2cm},
  }
  \pgfplotsset{
    scalaxis/.style={
      compat=newest,
      width=10.5cm,
      height=6.5cm, 
      xmin=48,
      xmax=29640,
      xtick={48, 456, 3744, 29640},
      grid=major,
      tick label style={font=\footnotesize},
      xticklabel style={align=center},
      title style={align=center, font=\footnotesize},
      xlabel style={align=center, font=\footnotesize},
      ylabel style={align=center, font=\footnotesize},
      legend style={at={(0.01,0.98)}, anchor=north west, font=\footnotesize},
      legend cell align=left},
    scalaxisscale/.style={
      /pgfplots/xmode=log,
      /pgfplots/ymode=log},
    scalaxisscale/.belongs to family=/pgfplots/scale,
  }
  \begin{subfigure}[b]{\columnwidth}
    \centering
    \begin{tikzpicture}
      \def\colorIdeal{jr@gray}
      \def\colorA{jr@blue}
      \def\colorB{jr@green}
      \begin{axis}[
        scalaxisscale,scalaxis,
        xlabel={Number of cores \minor{($\sim10^6$ DOF/core)}},
        ylabel={Degrees of freedom per unit time},
        ylabel style={align=center, font=\footnotesize,
                      at={(axis description cs:-0.1,0.44)}, anchor=south},
        xticklabels={\num{48}, \num{456}, \num{3744}, \num{29640}},
        ymin=1e6,
        ymax=1e11,
        ytick={1e6, 1e7, 1e8, 1e9, 1e10},
      ]
        \addplot[scalgraph=\colorIdeal] coordinates {
          (   48, 1.15e7)
          (29640, 7.10e9)
        };
        \addplot[scalgraph=\colorIdeal] coordinates {
          (   48, 1.51e6)
          (29640, 9.32e8)
        };
        \addplot[scalgraph=\colorA, mark=square*] coordinates {
          (   48, 1.15e7)
          (  456, 1.06e8)
          ( 3744, 8.49e8)
          (29640, 6.40e9)
        };
        \node [scallabeltop=\colorA] at (axis cs:   456, 1.06e8) {0.97};
        \node [scallabeltop=\colorA] at (axis cs:  3744, 8.49e8) {0.94};
        \node [scallabeltop=\colorA] at (axis cs: 29640, 6.40e9) {0.90};
        \addplot[scalgraph=\colorB, mark=triangle*] coordinates {
          (   48, 1.51e6)
          (  456, 1.39e7)
          ( 3744, 1.11e8)
          (29640, 6.65e8)
        };
        \node [scallabelbottom=\colorB] at (axis cs:   456, 1.39e7) {0.97};
        \node [scallabelbottom=\colorB] at (axis cs:  3744, 1.11e8) {0.94};
        \node [scallabelbottom=\colorB] at (axis cs: 29640, 6.65e8) {0.71};
        \legend{Ideal weak scalability, ,
                Solve [DOF/(sec/iter)],
                Setup [DOF/sec]}
      \end{axis}
    \end{tikzpicture}
    \hspace{5ex} %
    \caption{Weak scalability ($k=2$, $\ell=7,\ldots,10$)}
    \label{fig:scalability-weak}
  \end{subfigure}
  \\[5ex]
  \begin{subfigure}[b]{\columnwidth}
    \centering
    \begin{tikzpicture}
      \def\colorIdeal{jr@gray}
      \def\colorA{jr@purple}
      \def\colorC{jr@pink}
      \def\colorD{jr@orange}
      \def\colorE{jr@brown}
      \def\markA{*}
      \def\markC{triangle*}
      \def\markD{diamond*}
      \def\markE{square*}
      \begin{axis}[
        scalaxisscale,scalaxis,
        xlabel={Number of cores \minor{(DOF/core)}},
        ylabel={Speedup [baseline sec/sec]},
        ylabel style={align=center, font=\footnotesize,
                      at={(axis description cs:-0.1,0.5)}, anchor=south},
        xticklabels={
          \num{   48} \\ \minor{(\num{1235675})},
          \num{  456} \\ \minor{(\num{ 130071})},
          \num{ 3744} \\ \minor{(\num{  15842})},
          \num{29640} \\ \minor{(\num{   2009})}},
        ymin=1,
        ymax=2048,
        ytick={1, 8, 64, 512},
        yticklabels={1, 8, 64, 512},
      ]
        \addplot[scalgraph=\colorIdeal] coordinates {
          (   48,   1.0)
          (29640, 617.5)
        };
        \addplot[scalgraph=\colorA, mark=\markA] coordinates {
          (   48,  1.00)
          (  456,  9.48)
          ( 3744, 52.68)
          (29640, 33.01)
        };
        \node [scallabelbottom=\colorA] at (axis cs:   456,  9.48) {1.00};
        \node [scallabelbottom=\colorA] at (axis cs:  3744, 52.68) {0.68};
        \node [scallabelbottom=\colorA] at (axis cs: 29640, 33.01) {0.05};
        \addplot[scalgraph=\colorC, mark=\markC] coordinates {
          (   48,  1.00)
          (  456,  9.70)
          ( 3744, 57.42)
          (29640, 41.32)
        };
        \addplot[scalgraph=\colorD, mark=\markD] coordinates {
          (   48,  1.00)
          (  456, 10.31)
          ( 3744, 72.60)
          (29640, 64.74)
        };
        \addplot[scalgraph=\colorE, mark=\markE] coordinates {
          (   48,   1.00)
          (  456,  11.44)
          ( 3744,  83.92)
          (29640, 109.28)
        };
        \node [scallabeltop=\colorE] at (axis cs:   456,  11.44) {1.20};
        \node [scallabeltop=\colorE] at (axis cs:  3744,  83.92) {1.08};
        \node [scallabeltop=\colorE] at (axis cs: 29640, 109.28) {0.18};
        \legend{Ideal speedup,
                Solve 24$\times$OMP1,
                Solve 6$\times$OMP4,
                Solve 2$\times$OMP12,
                Solve 1$\times$OMP24}
      \end{axis}
    \end{tikzpicture}
    \hspace{5ex} %
    \caption{Strong scalability ($k=2$, $\ell=7$)}
    \label{fig:scalability-strong}
  \end{subfigure}
  \caption[Parallel scalability on Lonestar 5 for Stokes solver]{%
    Parallel scalability on Lonestar 5 for Stokes solver with \HMGwBFBT\
    preconditioning (problem \textup{S16-rand}, $\dynamicratio(\visc)=10^6$ as
    in \cref{tab:scalability-alg-level}).
    \subref{fig:scalability-weak} Weak scalability of setup and solve
    phases (normalized w.r.t.\ deviations from const.\ DOF/core).
    \infigure{Numbers along the graph lines} indicate weak parallel efficiency
    w.r.t.\ ideal weak scalability (baseline is 48 cores result).
    DOF/core is $\sim$1 million;
    the largest problem size on 29,640 cores has 30 billion DOF.
    \subref{fig:scalability-strong} Strong scalability of solve phase
    for different configurations of OpenMP threads (OMP) substituting
    MPI ranks on each node consisting of 24 cores.
    \infigure{Numbers along the graph lines} indicate strong efficiency w.r.t.\
    ideal speedup (baseline is 48 cores result).
  }
  \label{fig:scalability}
\end{figure}

\section{Conclusions}

For Stokes flow problems with highly heterogeneous viscosity, the
commonly used inverse viscosity-weighted mass matrix approximation of
the Schur complement can be insufficient. As a consequence,
convergence of Schur complement-based Krylov solvers
can be extremely slow. While each iteration with the weighted BFBT
(\wBFBT) Schur complement approximation proposed in this paper is
computationally more costly, we observe that it results in robust and
fast convergence even for complex viscosity structures and for up to
ten orders of magnitude viscosity contrasts---properties than occur,
for instance, in problems involving non-Newtonian geophysical
fluids.
Our numerical findings
are supported by theoretical spectral
equivalence results. To the best of our knowledge, a similar analysis
has not been shown for any \BFBT-type method before.

At Dirichlet boundaries, we use a modification of \wBFBT\ that is
necessary for mesh-independent convergence. In this modification, we
dampen the influence of elements at these boundaries, which would
otherwise dominate in \wBFBT\ and degrade its efficacy as a
preconditioner. Finding an alternative remedy is a subject of current
research.

Indefinite problems with highly heterogeneous coefficients and
high-order discretizations present significant challenges for
efficient solvers, especially in parallel. Nevertheless, we have
demonstrated that with careful attention paid to algorithm design and
scalable implementation, viscosity-robust and mesh-independent solvers
can be designed that exhibit nearly optimal algorithmic and parallel
scalability over a wide range of problems sizes and core counts.

\section*{Acknowledgments}

We wish to offer our deep thanks to the Texas Advanced Computing Center (TACC)
for granting us early user access to Lonestar~5 and enabling runs on the full
system.
We appreciate many helpful discussions with Tobin Isaac, who also
allowed us to build on his previous work on Stokes solvers.
We also thank Carsten Burstedde for his dedicated work on the p4est
library.

\bibliographystyle{plain}
\bibliography{ccgo}

\end{document}

%% file: colors_jr.tex
\definecolor{jr@darkred}{RGB}{153,0,0}
\definecolor{jr@medred}{RGB}{204,51,51}
\definecolor{jr@darkgreen}{RGB}{0,102,0}
\definecolor{jr@lightgreen}{RGB}{233,255,233}
\definecolor{jr@darkcyan}{RGB}{51,102,102}
\definecolor{jr@lightcyan}{RGB}{233,255,255}
\definecolor{jr@darkgray}{RGB}{51,51,51}
\definecolor{jr@medgray}{RGB}{102,102,102}
\definecolor{jr@lightgray}{RGB}{229,229,229}

\definecolor{jr@red}{RGB}{228,26,28}
\definecolor{jr@blue}{RGB}{55,126,184}
\definecolor{jr@green}{RGB}{77,175,74}
\definecolor{jr@purple}{RGB}{152,78,163}
\definecolor{jr@orange}{RGB}{255,127,0}
\definecolor{jr@yellow}{RGB}{255,255,51}
\definecolor{jr@brown}{RGB}{166,86,40}
\definecolor{jr@pink}{RGB}{247,129,191}
\definecolor{jr@gray}{RGB}{153,153,153}

\definecolor{jr@div-darkred}{RGB}{178,24,43}
\definecolor{jr@div-medred}{RGB}{239,138,98}
\definecolor{jr@div-lightred}{RGB}{253,219,199}
\definecolor{jr@div-lightblue}{RGB}{209,229,240}
\definecolor{jr@div-medblue}{RGB}{103,169,207}
\definecolor{jr@div-darkblue}{RGB}{33,102,172}

\definecolor{ices@orange}{RGB}{191,87,0}
\definecolor{ices@cyanblue}{RGB}{94,156,174}
\definecolor{ices@lightcyanblue}{RGB}{91,198,232}
\definecolor{ices@darkcyanblue}{RGB}{57,74,88}
\definecolor{ices@verydarkcyanblue}{RGB}{55,66,74}
\definecolor{ices@gray}{RGB}{200,201,199}
\definecolor{ices@green}{RGB}{114,206,155}
\definecolor{ices@yellow}{RGB}{242,169,0}

%% file: math_symbols.tex
\newcommand{\trans}{\mathsf{T}}
\newcommand{\transpose}{^{\trans}}
\newcommand{\adjoint}{^{\ast}}
\newcommand{\dual}{^{\prime}}

\newcommand{\gradient}{\nabla}
\newcommand{\symgradient}{\nabla_{\!\mathrm{s}}}
\newcommand{\divergence}{\nabla\cdot}

\newcommand{\intd}{\,\mathrm{d}}

\newcommand{\abs}[1]{\left|#1\right|}
\newcommand{\norm}[1]{\left\|#1\right\|}

\newcommand{\innerprod}[2]{\left(#1\mathbin{,}#2\right)}

\newcommand{\Lebesgue}[1]{\ensuremath{ L^{#1} }}
\newcommand{\LL}{\Lebesgue{2}}
\newcommand{\Linf}{\Lebesgue{\infty}}

\newcommand{\Sobolev}[1]{\ensuremath{ H^{#1} }}

\newcommand{\LandauO}{\ensuremath{\mathcal{O}}}

\DeclareMathOperator{\diag}{diag}

\DeclareMathOperator{\dynamicratio}{DR}

\DeclareMathOperator*{\esssup}{ess\,sup}

\newcommand{\set}[1]{\ensuremath{\mathbb{#1}}}

\newcommand{\setP}{\set{P}}
\newcommand{\Q}{\set{Q}}
\newcommand{\R}{\set{R}}

\newcommand{\boldvar}[1]{\ensuremath{\boldsymbol{#1}}}
\newcommand{\boldvec}[1]{\ensuremath{\mathbf{#1}}}
\newcommand{\mat}[1]{\boldvec{#1}}
\ifdefined\vec
\renewcommand{\vec}[1]{\boldvec{#1}}
\else
\newcommand{\vec}[1]{\boldvec{#1}}
\fi

\newcommand{\bn}{\boldvar{n}}
\newcommand{\bT}{\boldvar{T}}
\newcommand{\bx}{\boldvar{x}}

\newcommand{\unitvec}{\vec{e}}
\newcommand{\zerovec}{\vec{0}}
\newcommand{\onevec}{\vec{1}}
\newcommand{\IdMat}{\mat{I}}
\newcommand{\ZeroMat}{\mat{0}}

\newcommand{\vel}{\boldvar{u}}
\newcommand{\press}{p}

\newcommand{\visc}{\mu}

\newcommand{\viscstresstensor}{\boldvar{\tau}}

\newcommand{\rhsvel}{\boldvar{f}}

\newcommand{\viscmin}{\visc_{\mathrm{min}}}
\newcommand{\viscmax}{\visc_{\mathrm{max}}}

\newcommand{\FncSpVel}{(\Sobolev{1}(\Omega))^{d}}

\newcommand{\FncSpPress}{\Lebesgue{2}(\Omega)}
\newcommand{\FncSpPressMean}{\Lebesgue{2}(\Omega)/\R}
\newcommand{\FncSpVelGrad}{(\Lebesgue{2}(\Omega))^{d \times d}}
\newcommand{\FncSpVelDual}{(\Sobolev{-1}(\Omega))^{d}}
\newcommand{\FncSpPressHiReg}{\Sobolev{2}(\Omega)}

\newcommand{\ViscStressOp}{A}
\newcommand{\GradOp}{B\adjoint}
\newcommand{\DivOp}{B}
\newcommand{\GradOpExt}{B\adjoint} 
\newcommand{\DivOpExt}{B} 
\newcommand{\SchurOp}{S}
\newcommand{\PressPoissonOp}{K} 

\newcommand{\SchurApproxOpWBFBT}{\tilde{\SchurOp}_{\mathrm{w\text{-}BFBT}}}

\newcommand{\bfbtweight}{w}
\newcommand{\bfbtweightleft}{\bfbtweight_{l}}
\newcommand{\bfbtweightright}{\bfbtweight_{r}}

\newcommand{\discrQ}[1]{\ensuremath{ \Q_{#1} }}
\newcommand{\discrP}[1]{\ensuremath{ \setP_{#1} }}
\newcommand{\discrPdisc}[1]{\ensuremath{ \setP_{#1}^{\mathrm{disc}} }}
\newcommand{\discrPbubble}[1]{\ensuremath{ \setP_{#1}^{\mathrm{bubble}} }}
\newcommand{\discrQPdisc}[2]{\ensuremath{ \discrQ{#1}\times\discrPdisc{#2} }}
\newcommand{\discrQQ}[2]{\ensuremath{ \discrQ{#1}\times\discrQ{#2} }}
\newcommand{\discrPP}[2]{\ensuremath{ \discrP{#1}\times\discrP{#2} }}
\newcommand{\discrPbubblePdisc}[2]{\ensuremath{ \discrPbubble{#1}\times\discrPdisc{#2} }}

\newcommand{\velvec}{\vec{u}}
\newcommand{\pressvec}{\vec{p}}
\newcommand{\rhsvelvec}{\vec{f}}
\newcommand{\rhspressvec}{\vec{g}}

\newcommand{\ViscStressMat}{\mat{A}}
\newcommand{\GradMat}{\mat{B}\transpose}
\newcommand{\DivMat}{\mat{B}}
\newcommand{\SchurMat}{\mat{S}}
\newcommand{\ViscStressPcMat}{\tilde{\ViscStressMat}}
\newcommand{\SchurPcMat}{\tilde{\SchurMat}}

\newcommand{\MassMat}{\mat{M}}
\newcommand{\MassLumpMat}{\tilde{\mat{M}}}
\newcommand{\VelMassMat}{\MassMat_{\vel}}
\newcommand{\VelMassLumpMat}{\MassLumpMat_{\vel}}
\newcommand{\PressMassMat}{\MassMat_{\press}}
\newcommand{\PressMassLumpMat}{\MassLumpMat_{\press}}

\newcommand{\BfbtScalLeftMat}{\mat{C}}
\newcommand{\BfbtScalRightMat}{\mat{D}}
\newcommand{\WbfbtScalLeftMat}{\BfbtScalLeftMat_{\bfbtweightleft}}
\newcommand{\WbfbtScalRightMat}{\BfbtScalRightMat_{\bfbtweightright}}
\newcommand{\PressPoissonMat}{\mat{K}}
\newcommand{\PressPoissonLeftMat}{\PressPoissonMat_{\bfbtweightleft}}
\newcommand{\PressPoissonRightMat}{\PressPoissonMat_{\bfbtweightright}}
\newcommand{\PressPoissonLeftPcMat}{\tilde{\PressPoissonMat}_{\bfbtweightleft}}
\newcommand{\PressPoissonRightPcMat}{\tilde{\PressPoissonMat}_{\bfbtweightright}}

\newcommand{\veltest}{\boldvar{v}}
\newcommand{\presstest}{q}



%% file: inc/robustness_eigvals_1d_style.tex
%
%
\begin{tikzpicture}
  \tikzset{
    spectrumgraph/.style={
      color=#1, line width=2pt, opacity=0.5,
      mark size=6pt, mark options={line width=1.0pt,opacity=0.5}},
  }
  \pgfplotsset{
    spectrumaxis/.style={
      compat=newest,
      width=7.2cm,
      height=5.2cm,
      xmin=1e-8,
      xmax=1e+3,
      ymin=-0.6,
      ymax=+2.6,
      xtick={1e2, 1e0, 1e-2, 1e-4, 1e-6, 1e-8},
      ytick={0, 1, 2},
      ymajorticks=false,
      xmajorgrids=true,
      tick label style={font=\footnotesize},
      title style={align=center, font=\footnotesize, yshift=-0.2cm},
      xlabel style={align=center, font=\footnotesize},
      ylabel style={align=center, font=\footnotesize},
      legend style={at={(0.01,0.988)}, anchor=north west, font=\footnotesize},
      legend cell align=left,
      reverse legend},
  }
  \pgfplotsset{
    spectrumaxisscale/.style={
      /pgfplots/xmode=log,
      /pgfplots/ymode=linear},
    spectrumaxisscale/.belongs to family=/pgfplots/scale,
  }

  \def\colorA{jr@div-medblue!40!jr@div-darkblue}
  \def\colorB{jr@div-medred!90!jr@div-darkred}
  \def\colorC{jr@lightgray!20!jr@medgray}
  \def\markA{|}
  \def\markB{|}
  \def\markC{|}
  \def\legendA{$\SchurPcMat_{\wBFBT}^{-1} \, \SchurMat$}
  \def\legendB{$\PressMassMat^{-1}(1/\visc) \, \SchurMat$}
  \def\legendC{$\SchurMat$}
  \def\datadir{inc/plotdata/robustness_eigvals}

  \newcommand{\addplotspectrum}[4]{%
    \addplot[spectrumgraph=#1, mark=#2,
             legend image post style={line width=2pt,opacity=1,mark=none}]
      table [x=eigenvalue,y=label] {#3};
    \addlegendentry{#4}
  }

  \def\dataA{\datadir/Srand1_viscDrMagn4_l4_PCxS_newsol.txt}
  \def\dataB{\datadir/Srand1_viscDrMagn4_l4_PCxS_refsol.txt}
  \def\dataC{\datadir/Srand1_viscDrMagn4_l4_S.txt}
  \begin{axis}[
    name=S1rand-DR1e4,
    spectrumaxis, spectrumaxisscale,
    ylabel={S1-rand, $\dynamicratio(\visc)=10^4$},
  ]
    \addplotspectrum{\colorC}{\markC}{\dataC}{\legendC}
    \addplotspectrum{\colorB}{\markB}{\dataB}{\legendB}
    \addplotspectrum{\colorA}{\markA}{\dataA}{\legendA}
  \end{axis}

  \def\dataA{\datadir/Srand1_viscDrMagn6_l4_PCxS_newsol.txt}
  \def\dataB{\datadir/Srand1_viscDrMagn6_l4_PCxS_refsol.txt}
  \def\dataC{\datadir/Srand1_viscDrMagn6_l4_S.txt}
  \begin{axis}[
    at=(S1rand-DR1e4.right of south east),
    anchor=left of south west,
    xshift=4ex,
    name=S1rand-DR1e6,
    spectrumaxis, spectrumaxisscale,
    ylabel={S1-rand, $\dynamicratio(\visc)=10^6$},
  ]
    \addplotspectrum{\colorC}{\markC}{\dataC}{\legendC}
    \addplotspectrum{\colorB}{\markB}{\dataB}{\legendB}
    \addplotspectrum{\colorA}{\markA}{\dataA}{\legendA}
  \end{axis}

  \def\dataA{\datadir/Srand4_viscDrMagn4_l4_PCxS_newsol.txt}
  \def\dataB{\datadir/Srand4_viscDrMagn4_l4_PCxS_refsol.txt}
  \def\dataC{\datadir/Srand4_viscDrMagn4_l4_S.txt}
  \begin{axis}[
    at=(S1rand-DR1e4.left of south west),
    anchor=left of north west,
    yshift=-5ex,
    name=S4rand-DR1e4,
    spectrumaxis, spectrumaxisscale,
    xlabel={Eigenvalue},
    ylabel={S4-rand, $\dynamicratio(\visc)=10^4$},
  ]
    \addplotspectrum{\colorC}{\markC}{\dataC}{\legendC}
    \addplotspectrum{\colorB}{\markB}{\dataB}{\legendB}
    \addplotspectrum{\colorA}{\markA}{\dataA}{\legendA}
  \end{axis}

  \def\dataA{\datadir/Srand4_viscDrMagn6_l4_PCxS_newsol.txt}
  \def\dataB{\datadir/Srand4_viscDrMagn6_l4_PCxS_refsol.txt}
  \def\dataC{\datadir/Srand4_viscDrMagn6_l4_S.txt}
  \begin{axis}[
    at=(S4rand-DR1e4.right of south east),
    anchor=left of south west,
    xshift=4ex,
    name=S4rand-DR1e6,
    spectrumaxis, spectrumaxisscale,
    xlabel={Eigenvalue},
    ylabel={S4-rand, $\dynamicratio(\visc)=10^6$},
  ]
    \addplotspectrum{\colorC}{\markC}{\dataC}{\legendC}
    \addplotspectrum{\colorB}{\markB}{\dataB}{\legendB}
    \addplotspectrum{\colorA}{\markA}{\dataA}{\legendA}
  \end{axis}
\end{tikzpicture}%